\documentclass[11pt,reqno, notitlepage]{amsart}
\usepackage{amsmath,amsfonts,amsthm,amssymb}
\usepackage{graphicx}
\usepackage{verbatim}
\usepackage{color}
\usepackage{amscd}
\usepackage{verbatim}

\usepackage[french,english]{babel}

\usepackage[utf8]{inputenc}

\begin{document}
\newcommand{\M}{{\mathcal M}}
\newcommand{\loc}{{\mathrm{loc}}}
\newcommand{\core}{C_0^{\infty}(\Omega)}
\newcommand{\sob}{W^{1,p}(\Omega)}
\newcommand{\sobloc}{W^{1,p}_{\mathrm{loc}}(\Omega)}
\newcommand{\merhav}{{\mathcal D}^{1,p}}
\newcommand{\be}{\begin{equation}}
\newcommand{\ee}{\end{equation}}
\newcommand{\mytion}[1]{\section{#1}\setcounter{equation}{0}}
\newcommand{\laplace}{\Delta}
\newcommand{\pl}{\laplace_p}
\newcommand{\grad}{\nabla}
\newcommand{\pd}{\partial}
\newcommand{\bo}{\pd}
\newcommand{\csub}{\subset \subset}
\newcommand{\sm}{\setminus}
\newcommand{\ssm}{:}
\newcommand{\diver}{\mathrm{div}\,}
\newcommand{\bea}{\begin{eqnarray}}
\newcommand{\eea}{\end{eqnarray}}
\newcommand{\bean}{\begin{eqnarray*}}
\newcommand{\eean}{\end{eqnarray*}}
\newcommand{\thkl}{\rule[-.5mm]{.3mm}{3mm}}
\newcommand{\cw}{\stackrel{\rightharpoonup}{\rightharpoonup}}
\newcommand{\id}{\operatorname{id}}
\newcommand{\supp}{\operatorname{supp}}
\newcommand{\wlim}{\mbox{ w-lim }}
\newcommand{\mymu}{{x_N^{-p_*}}}
\newcommand{\R}{{\mathbb R}}
\newcommand{\N}{{\mathbb N}}
\newcommand{\Z}{{\mathbb Z}}
\newcommand{\Q}{{\mathbb Q}}
\newcommand{\abs}[1]{\lvert#1\rvert}
\newtheorem{theorem}{Theorem}[section]
\newtheorem{corollary}[theorem]{Corollary}
\newtheorem{lemma}[theorem]{Lemma}
\newtheorem{notation}[theorem]{Notation}
\newtheorem{definition}[theorem]{Definition}
\newtheorem{remark}[theorem]{Remark}
\newtheorem{proposition}[theorem]{Proposition}
\newtheorem{assertion}[theorem]{Assertion}
\newtheorem{problem}[theorem]{Problem}
\newtheorem{conjecture}[theorem]{Conjecture}
\newtheorem{question}[theorem]{Question}
\newtheorem{example}[theorem]{Example}
\newtheorem{Thm}[theorem]{Theorem}
\newtheorem{Lem}[theorem]{Lemma}
\newtheorem{Pro}[theorem]{Proposition}
\newtheorem{Def}[theorem]{Definition}
\newtheorem{defi}[theorem]{Definition}
\newtheorem{Exa}[theorem]{Example}
\newtheorem{Exs}[theorem]{Examples}
\newtheorem{Rems}[theorem]{Remarks}
\newtheorem{Rem}[theorem]{Remark}

\newtheorem{Cor}[theorem]{Corollary}
\newtheorem{Conj}[theorem]{Conjecture}
\newtheorem{Prob}[theorem]{Problem}
\newtheorem{Ques}[theorem]{Question}
\newtheorem*{corollary*}{Corollary}
\newtheorem*{theorem*}{Theorem}
\newtheorem{thm}[theorem]{Theorem}
\newtheorem{lem}[theorem]{Lemma}
\newtheorem{prop}[theorem]{Proposition}
\newtheorem{cor}[theorem]{Corollary}
\newtheorem{ex}[theorem]{Example}
\newtheorem{rem}[theorem]{Remark}
\newtheorem{rems}[theorem]{Remarks}
\newtheorem*{thmm}{Theorem}
\newcommand{\Hmm}[1]{\leavevmode{\marginpar{\tiny%
$\hbox to 0mm{\hspace*{-0.5mm}$\leftarrow$\hss}%
\vcenter{\vrule depth 0.1mm height 0.1mm width \the\marginparwidth}%
\hbox to
0mm{\hss$\rightarrow$\hspace*{-0.5mm}}$\\\relax\raggedright #1}}}
\newcommand{\pf}{\noindent \mbox{{\bf Proof}: }}


\renewcommand{\theequation}{\thesection.\arabic{equation}}
\catcode`@=11 \@addtoreset{equation}{section} \catcode`@=12
\newcommand{\Real}{\mathbb{R}}
\newcommand{\real}{\mathbb{R}}
\newcommand{\Nat}{\mathbb{N}}
\newcommand{\ZZ}{\mathbb{Z}}
\newcommand{\CC}{\mathbb{C}}
\newcommand{\Pess}{\opname{Pess}}
\newcommand{\Proof}{\mbox{\noindent {\bf Proof} \hspace{2mm}}}
\newcommand{\mbinom}[2]{\left (\!\!{\renewcommand{\arraystretch}{0.5}
\mbox{$\begin{array}[c]{c}  #1\\ #2  \end{array}$}}\!\! \right )}
\newcommand{\brang}[1]{\langle #1 \rangle}
\newcommand{\vstrut}[1]{\rule{0mm}{#1mm}}
\newcommand{\rec}[1]{\frac{1}{#1}}
\newcommand{\set}[1]{\{#1\}}
\newcommand{\dist}[2]{$\mbox{\rm dist}\,(#1,#2)$}
\newcommand{\opname}[1]{\mbox{\rm #1}\,}
\newcommand{\mb}[1]{\;\mbox{ #1 }\;}
\newcommand{\undersym}[2]
 {{\renewcommand{\arraystretch}{0.5}  \mbox{$\begin{array}[t]{c}
 #1\\ #2  \end{array}$}}}
\newlength{\wex}  \newlength{\hex}
\newcommand{\understack}[3]{%
 \settowidth{\wex}{\mbox{$#3$}} \settoheight{\hex}{\mbox{$#1$}}
 \hspace{\wex}  \raisebox{-1.2\hex}{\makebox[-\wex][c]{$#2$}}
 \makebox[\wex][c]{$#1$}   }%
\newcommand{\smit}[1]{\mbox{\small \it #1}}
\newcommand{\lgit}[1]{\mbox{\large \it #1}}
\newcommand{\scts}[1]{\scriptstyle #1}
\newcommand{\scss}[1]{\scriptscriptstyle #1}
\newcommand{\txts}[1]{\textstyle #1}
\newcommand{\dsps}[1]{\displaystyle #1}
\newcommand{\dx}{\,\mathrm{d}x}
\newcommand{\dy}{\,\mathrm{d}y}
\newcommand{\dz}{\,\mathrm{d}z}
\newcommand{\dm}{\,\mathrm{d}m}
\newcommand{\dt}{\,\mathrm{d}t}
\newcommand{\dr}{\,\mathrm{d}r}
\newcommand{\du}{\,\mathrm{d}u}
\newcommand{\dv}{\,\mathrm{d}v}
\newcommand{\dV}{\,\mathrm{d}V}
\newcommand{\ds}{\,\mathrm{d}s}
\newcommand{\da}{\,\mathrm{d}\alpha}
\newcommand{\db}{\,\mathrm{d}\beta}
\newcommand{\dS}{\,\mathrm{d}S}
\newcommand{\dk}{\,\mathrm{d}k}

\newcommand{\dphi}{\,\mathrm{d}\phi}
\newcommand{\dtau}{\,\mathrm{d}\tau}
\newcommand{\dxi}{\,\mathrm{d}\xi}
\newcommand{\deta}{\,\mathrm{d}\eta}
\newcommand{\dsigma}{\,\mathrm{d}\sigma}
\newcommand{\dtheta}{\,\mathrm{d}\theta}
\newcommand{\dnu}{\,\mathrm{d}\nu}

\def\ga{\alpha}     \def\gb{\beta}       \def\gg{\gamma}
\def\gc{\chi}       \def\gd{\delta}      \def\ge{\epsilon}
\def\gth{\theta}                         \def\vge{\varepsilon}
\def\gf{\phi}       \def\vgf{\varphi}    \def\gh{\eta}
\def\gi{\iota}      \def\gk{\kappa}      \def\gl{\lambda}
\def\gm{\mu}        \def\gn{\nu}         \def\gp{\pi}
\def\vgp{\varpi}    \def\gr{\rho}        \def\vgr{\varrho}
\def\gs{\sigma}     \def\vgs{\varsigma}  \def\gt{\tau}
\def\gu{\upsilon}   \def\gv{\vartheta}   \def\gw{\omega}
\def\gx{\xi}        \def\gy{\psi}        \def\gz{\zeta}
\def\Gg{\Gamma}     \def\Gd{\Delta}      \def\Gf{\Phi}
\def\Gth{\Theta}
\def\Gl{\Lambda}    \def\Gs{\Sigma}      \def\Gp{\Pi}
\def\Gw{\Omega}     \def\Gx{\Xi}         \def\Gy{\Psi}

\renewcommand{\div}{\mathrm{div}}
\newcommand{\red}[1]{{\color{red} #1}}

%


\newcommand{\De} {\Delta}
\newcommand{\la} {\lambda}
\newcommand{\bn}{\mathbb{B}^{2}}
\newcommand{\rn}{\mathbb{R}^{2}}
\newcommand{\bnn}{\mathbb{B}^{N}}
\newcommand{\rnn}{\mathbb{R}^{N}}
\newcommand{\kP}{k_{P}^{M}}
\newcommand{\authorfootnotes}{\renewcommand\thefootnote{\@fnsymbol\c@footnote}}%

\def\e{{\text{e}}}
\def\N{{I\!\!N}}

\numberwithin{equation}{section} \allowdisplaybreaks

\title[equivalence of heat kernels]{On the equivalence of heat kernels of second-order parabolic operators}

\author{Debdip Ganguly}
\address{Debdip Ganguly, Department of Mathematics,  Technion - Israel Institute of Technology, Haifa 32000, Israel}
\email{gdebdip@technion.ac.il}
\author{Yehuda Pinchover}
\address{Yehuda Pinchover,
Department of Mathematics, Technion - Israel Institute of
Technology,   Haifa 32000, Israel}
\email{pincho@technion.ac.il}

\date{}


\begin{abstract}
Let $P$ be a second-order, symmetric, and nonnegative elliptic operator with real coefficients defined on noncompact Riemannian manifold $M$, and let $V$ be a real valued function which belongs to the class of {\em small perturbation potentials} with respect to the heat kernel of $P$ in $M$.  We prove that under some further assumptions (satisfying by a large classes of $P$ and $M$) the positive minimal heat kernels of $P-V$ and of $P$ on $M$ are equivalent. Moreover, the parabolic Martin boundary is stable under such perturbations, and the cones of all nonnegative solutions of the corresponding parabolic equations are affine homeomorphic.

\vspace{.2cm}

\noindent  2000  \! {\em Mathematics  Subject  Classification.}
{Primary 35K08; Secondary 31C35, 35B09, 47D07, 47D08} \\[1mm]
\end{abstract}

\maketitle


 \section{Introduction}\label{sec_int}
Let $M$ be a smooth, noncompact, connected Riemannian manifold of dimension $N$.  Let $P$ be a second-order
elliptic linear operator defined on $M$, and let $V$ be a real valued potential. Denote the cone of all positive solutions of the
equation $Pu=0$ in $M$ by $\mathcal{C}_{P}(M)$. The {\em
generalized principal eigenvalue} of the operator $P$ and a potential $V$  is defined by
$$\gl_0(P,V,M)
:= \sup\{\gl \in \mathbb{R} \; \mid\; \mathcal{C}_{P-\lambda V}(M)\neq
\emptyset\}.$$
We say that $P$ is {\em nonnegative in} $M$ (and denote it by $P\geq 0$) if $\lambda_0:= \lambda_0(P,\mathbf{1},M)\geq 0$, where $\mathbf{1}$ is the constant function on $M$ taking
at any point $x\in M$ the value $1$. Throughout the paper we always assume that
$\gl_0\geq 0$, that is, $P\geq 0$ in $M$.
So, let $P\geq 0$ in $M$, and consider the parabolic operator
\begin{equation}\label{eqL}
  Lu:=\pd_t u+Pu \qquad  (x,t)\in M\times (0,\infty).
\end{equation}
Let $k_P^{M}(x,y,t)$ be the {\em positive minimal (Dirichlet) heat kernel} of the
parabolic operator $L$ on the manifold $M$.  By definition, for a fixed $y\in M$, the function $(x,t) \mapsto k_P^M(x,y,t)$
is the minimal positive solution of the equation
\begin{equation}\label{heat}
 Lu =0  \qquad\mbox{ in } M\times(0,\infty) ,
\end{equation}
subject to  the initial data $\delta_y$, the Dirac distribution at $y \in M$.


Let $g_1,g_2$ be two positive functions defined in a domain $D$. We say that $g_1$ is {\em equivalent} to $g_2$ in $D$ (and use the notation $g_1\asymp g_2$ in
$D$) if there exists a positive constant $C$ such
that
$$C^{-1}g_{2}(x)\leq g_{1}(x) \leq Cg_{2}(x) \qquad \mbox{ for all } x\in D.$$

The main aim of this article is to study the equivalence of two heat kernels associated with two parabolic operators in $M$.  We are motivated by the following conjecture raised in \cite{FKP}.
\begin{conjecture}[cf. \cite{FKP}]\label{conjequival}
Let $P_1$ and $P_2$ be two subcritical elliptic operators either of the form \eqref{P} or \eqref{operator} which are
defined on a Riemannian manifold $M$ such that both  $P_1$ and $P_2$ have same principle part. Assume that $P_1=P_2$ outside a compact set in $M$ and that the generalized principal eigenvalues
$\lambda_{0}(P_1, \mathbf{1}, M),$ $\lambda_{0}(P_2, \mathbf{1}, M)$  of $P_1$ and $P_2$ respectively in $M$ are equal.  Then $k_{P_1}^M \asymp
k_{P_2}^M$ in $M\times M\times (0,\infty)$.
\end{conjecture}
An important aspect of Conjecture~\ref{conjequival} is towards the understanding the stability of the large time behaviour of heat kernels, and of the parabolic Martin boundary under perturbations. We also remark that Conjecture~\ref{conjequival} is related to strong ratio limit properties of the quotients of heat
kernels of subcritical and critical operators, and to Davies' Conjecture (see \cite{FKP}).

In the past four decades there has been an extensive research in obtaining optimal sufficient conditions under which two second-order {\em elliptic} operators have equivalent positive Green functions, and the elliptic case is pretty much well understood (see for example \cite{AA,MM1,YP3,YP1,YP2}, and references therein).  On the other hand, in spite of the huge literature dealing with two-sided heat kernel {\em estimates}, the question of the {\em equivalence} of heat kernels is far from being understood. In fact, there are only very few papers dealing with sufficient conditions that guarantee the equivalence of the heat kernels; see \cite{BDS,BS,FKP,LS,MS,Z1,Z2}. Moreover, most of these works study the particular case of a perturbation of the Laplace operator on $\mathbb{R}^N$ by a potential $V$ that is either a {\em signed} potential, or satisfies additional smoothness assumptions.

 Note that the explicit form of the heat kernel of the Laplacian on $\mathbb{R}^N$ is given by the Gauss-Weierstrass heat kernel
\begin{equation}\label{GW}
    k_{-\Delta}^{\mathbb{R}^{N}}(x, y, t) := \left(\frac{1}{4 \pi t} \right)^{\frac{N}{2}} \mathrm{e}^{-\frac{|x - y|^2}{4t}} \qquad x, y \in \mathbb{R}^{N}, \, t>0,
\end{equation}
and this explicit formula plays a crucial role in almost all the aforementioned  papers (except \cite{FKP}). Unfortunately, for  general operators and manifolds such an expression is not available, despite the fact that in many cases the short and large time behaviour of the heat kernel is known.
 However, we should mention the very recently paper by Chen and Hassell \cite{CH}, where it is proved that under natural assumptions, the heat kernel of an asymptotically hyperbolic  Cartan-Hadamard manifold, is equivalent to the heat kernel of the hyperbolic space.

We provide a positive answer to Conjecture~\ref{conjequival} in the case where $P$ is symmetric and satisfies some further assumptions. We prove in Theorem~\ref{main} the equivalence of two heat kernels of two parabolic operators that differ by a compactly supported potential. This result is extended in Theorem~\ref{spmain} to a larger class of potentials known as the class of small perturbations with respect to the given heat kernel (see Definition~\ref{def_sp}). As an application we prove that the parabolic Martin boundary is stable under such perturbations, and the cones of all nonnegative solutions of the corresponding parabolic equations are {\em affine homeomorphic}.

Our study is based on the method used by M.~Murata and Y.~P. in the study of the equivalence of the {\em Green functions} of elliptic operators (see \cite{MM1,YP3,YP1}).
In this approach one should obtain pointwise estimates for the iterated Green kernel, called the $3G$-inequality which implies sharp two-sided estimates for the corresponding Neumann series.
To understand the difficulty in applying this method to the parabolic case, assume for simplicity that $V$ has a compact support in $M$. In contrast to the elliptic case  \cite{MM1,YP3}, where the iterated kernel is given by integrations over a fixed compact set ($\supp V$), in the parabolic case the domain of integration is $\supp[V\times (0,t)]$ which grows as $t\to\infty$. Hence, the parabolic case  requires a new and a different technique in order to prove the so called {\em $3k$-inequality}. We refer to Section~\ref{secpreliminaries} for the definition of the {\em $3k$-inequality.}

The paper is organized as follows. In Section~\ref{sec_setting} we briefly review the theory of positive solutions of elliptic and parabolic equations and state our main results. Section~\ref{secpreliminaries} is devoted to several preparatory lemmas and propositions. In Section~\ref{secproof} we prove the aforementioned Theorem~\ref{main} concerning compactly supported perturbations, while in Section~\ref{smp} we introduce the notion of small perturbations with respect to the  given heat kernel and prove the aforementioned Theorem~\ref{spmain}. Section~\ref{stableMartin} is devoted to the stability of the  Martin boundary under small perturbations.  We conclude our paper in Section~\ref{concludingrem} which is divided into three short subsections. In the first subsection we briefly extend our results to the class of {\em quasi-symmetric} heat kernels,  in the second part we present some examples of manifolds and operators for which our results applies, and finally, a subsection devoted to a short discussion of some open problems ends the paper.
\section{The setting and statements of the main results}\label{sec_setting}
The present section is devoted to the statements of our main theorems. Before going further we must introduce some notations, technical assumptions and definitions.

Let $M$ be a smooth, noncompact, connected manifold of dimension $N$. We consider a second-order elliptic operator $P$ with real coefficients which (in any coordinate system $(U;x_{1},\ldots,x_{N})$) is either of the form
\be \label{P}
Pu=-\sum_{i=1}^N a^{ij}(x)\partial_{i}\partial_{j}u + b(x)\cdot\nabla u+c(x)u,
\end{equation}
or in the divergence form
\begin{equation} \label{operator}
Pu=-\div \left[\big(A(x)\nabla u +  u\tilde{b}(x) \big) \right]  +
 b(x)\cdot\nabla u   +c(x)u.
\end{equation}
We assume that for every $x\in\Gw$ the matrix $A(x):=\big[a^{ij}(x)\big]$ is symmetric and that the real quadratic form
\begin{equation*}
 \xi \cdot A(x) \xi := \sum_{i,j =1}^N \xi_i a^{ij}(x) \xi_j \qquad
 \xi \in \Real ^N
\end{equation*}
is positive definite. Moreover, it is assumed that $P$ is locally uniformly elliptic. Hence, the principal part of the operator $P$ induces a Riemannian metric $\mathfrak{g}$ on $M$. Throughout the paper we consider the Riemannian manifold $(M,\mathfrak{g})$.
 In particular, when $P= -\Delta_{\mathfrak{h}}$, is the Laplace-Beltrami on a given Riemannian Manifold $(M, \mathfrak{h})$, then the induced metric $\mathfrak{g}$ on $M$ coincides with the given metric $\mathfrak{h}.$
We assume that $\dx$ is a given positive measure on $M$, satisfying $\dx=f\,\mathrm{vol}$, where $f$ is a positive function,  and $\mathrm{vol}$ is the Riemannian volume form of $M$
with respect to the metric $\mathfrak{g}$  (which is just the Lebesgue measure in the case of a domain of $\R^N$ and the operator $P = - \Delta,$ Euclidean Laplacian).
Further, the minus divergence is the formal adjoint of the gradient with respect to the measure $\dx$.

Throughout the paper we assume that the coefficients of $P$ are either $C^\infty$-smooth or locally sufficiently regular in $M$ such that the standard parabolic regularity theory holds true.
For example, such sufficient conditions for $P$ of the form \eqref{operator} are: $f$  and $A$ are locally H\"{o}lder continuous,
the vector fields $b$ and $\tilde{b}$ are Borel measurable in $M$ of class
$L^p_{\mathrm{loc}}(M)$, and $c \in L^{p/2}_{\mathrm{loc}}(M)$ for some $p > N$. We denote by $P^\star$ the formal adjoint operator of $P$ on its natural space $L^2(M, \dx)$.

When $P$ is in divergence form (\ref{operator}) and $b = \tilde{b}$, the operator
\be\label{symm_P}
Pu = - \div \left[ \big(A \grad u + u b\big) \right] + b \cdot \grad u + c u,
\ee
is {\em symmetric} in the space $L^2(M, \dx)$. Throughout the paper, we call this setting the {\em symmetric case}.  We note that if $P$ is symmetric and $b$ is smooth enough, then $P$ is in fact a Schr\"odinger-type operator of the form
\be
\nonumber
Pu = - \div \big(A \grad u \big) + \big(c-\div b\big)u.
\ee
Assume that $\lambda_{0} \geq 0$, and let $k_P^{M}(x,y,t)$ be the {\em positive (minimal) heat kernel} of the
parabolic operator $L$ on the manifold $M$.
It can be easily checked that for $\lambda\leq \lambda_0$, the
heat kernel $k_{P-\lambda}^M$ of the operator $P-\lambda$ on $M$ satisfies the identity
\begin{equation}\label{eq_hkp-gl}
k_{P-\lambda}^M(x,y,t)=\mathrm{e}^{\lambda t}k_P^M(x,y,t)\qquad \mbox{ on }
M\times M\times (0,\infty).
\end{equation}
\begin{definition}\label{def_crit}{\em
Suppose that $\gl_0=\lambda_0(P,\mathbf{1},M)\geq 0$, and let $\kP$ be the corresponding heat kernel. We say that the operator~$P$ is \emph{subcritical} (respectively, \emph{critical}) in $M$ if for some $x \not = y$, (and therefore for any
$x \not = y$),  $x,y\in M$, we have
\begin{equation}\label{def.critical}
  \int_0^\infty k_P^{M}(x,y,\tau)\,\mathrm{d}\tau<\infty \quad
  \left(\mbox{respectively, } \int_0^\infty
  k_P^{M}(x,y,\tau)\,\mathrm{d}\tau=\infty\right).
\end{equation}
If $P$ is subcritical in $M$, then
\begin{equation}\label{G_k}
    G_P^{M}(x,y):=\int_0^\infty k_P^{M}(x,y,\tau)\,\mathrm{d}\tau \qquad x,y \in M
\end{equation}
is called the {\em positive minimal Green function}  of the operator $P$ in $M$.
}
\end{definition}
Clearly, $P$ is critical in $M$ if and only if $P^\star$ is critical in $M$. Moreover, it is known that $P$ is critical in $M$ if and only if the equation $Pu=0$ in $M$ admits a unique (up to a multiplicative constant) positive supersolution \cite{MM1,YP3,pinch2}. In this case
the corresponding unique (super)solution of the equation $Pu=0$ in $M$ is called the \emph{(Agmon) ground state}.

Suppose that $P$ is a critical operator in $M$ and let $\phi$ and $\phi^{*}$ be the ground states of $P$ and $P^\star$, respectively. $P$ is said to be \emph{positive-critical} (\emph{null-critical}) in $M$ with respect to the measure $\dx$ if  $\phi^{*} \phi \in L^{1}(M,\!\dx)$ (resp., $\phi^{*} \phi \not\in L^{1}(M,\!\dx))$.

\begin{rem}\label{rem}{\em   We recall some general results concerning the large time behaviour of the heat kernel.

Let $P$ be an elliptic operator either of the form \eqref{P} or \eqref{operator}, and assume that $\gl_0=\gl_0(P,\mathbf{1},M)\geq 0$.
Then
\begin{equation}\label{large-eigen}
-\lim_{t \rightarrow \infty} \frac{\log\kP(x, y, t)}{t}= \lambda_{0}.
\end{equation}
(see \cite[Remark~4]{FKP}, and references therein). Moreover,
$$\lim_{t \rightarrow \infty} \mathrm{e}^{\lambda_{0} t} \kP(x, y, t) = 0\quad  \mbox{locally uniformly in $M\times M$},$$ unless $P-\gl_0$ is positive-critical, and in this case,
\[
\lim_{t \rightarrow \infty} \mathrm{e}^{\lambda_{0}t} \kP(x, y, t) = \frac{\phi(x) \phi^*(y)}{\int_{M} \phi^{*}(z) \phi(z) \dz}
\]
locally uniformly in $M\times M$, where  $\phi$ and $\phi^{*}$ are the ground states of $P - \lambda_{0}$ and $P^{*} - \lambda_{0}$, respectively (see \cite[Theorem~1.2]{YP4}, and references therein).
 }
\end{rem}

 \begin{definition}\label{defequivalent}{\em
Let $P_{i},\,i=1,2$, be two elliptic operators  either of the form \eqref{P} or \eqref{operator}  that are defined on
$M$, and suppose that $\lambda_0(P_i,\mathbf{1},M)\geq 0$ for $i=1,2$. We say that the corresponding heat kernels
$k_{P_1}^M(x,y,t)$ and $k_{P_2}^M(x,y,t)$ are
{\em equivalent} (respectively, {\em semi-equivalent}) if $$k_{P_1}^M \asymp
k_{P_2}^M \qquad   \mbox{ on } M \times M \times (0,\infty)$$
$$ \mbox{(resp., } \;k_{P_1}^M (\cdot,y_0,\cdot)\asymp k_{P_2}^M(\cdot,y_0,\cdot) \mbox{ on }
\;M\times (0,\infty)  \mbox{  for some fixed }\; y_0\in M \mbox{)}.$$

Similarly, we define the equivalence and the semi-equivalence of the Green functions $G_{P_i}^M(x,y)$, where $i=1,2$.}
\end{definition}
\begin{rem}
{\em
It follows that if $k_{P_1}^M \asymp k_{P_2}^M$, then $P_1$ is subcritical in $M$ if and only if $P_2$ is subcritical in $M$, and in this case, \eqref{eq_hkp-gl} and \eqref{G_k} imply that
$G_{P_1-\gl}^M\asymp G_{P_2-\gl}^M$ for any $\gl\leq \gl_0$ with the same equivalence constant. Moreover,
$\lambda_{0}(P_1,\mathbf{1}, M) := \lambda_{0}(P_2,\mathbf{1}, M)$.
 }
\end{rem}
Throughout the paper we consider a perturbation of an elliptic operator $P$  by a real valued potential $V$. We introduce the following one-parameter family of operators
\begin{equation}\label{operator2}
P_{\varepsilon} := P -  \varepsilon V \qquad \vge \in \mathbb{R},
\end{equation}
where $P$ is a given elliptic operator either of the form \eqref{P} or \eqref{operator}  , and $V$ is a given potential satisfying the above regularity assumption.


Now we are in a situation to state the main results of the paper. In fact, we provide a positive answer to Conjecture~\ref{conjequival} under further assumptions.
\begin{thm}\label{main}
Let $(M, \mathfrak{g})$ be a connected and noncompact Riemannian manifold of dimension $N$.  Let $P$ be a symmetric subcritical operator with  $C^\infty$-coefficients, such that the induced Riemannian metric by $P$ is equal to $\mathfrak{g}$. Let $V \in L^{p}_{\loc}(M)$ be a nonzero real valued potential with compact support, where $ p > \frac{N}{2}$.

Assume that for some $x_0\in M$ and $T>0$ there exists $C:=C(T,x_0) > 0$ such that the following doubling condition holds
\begin{equation}\label{assumption1}
k_P^{M}(x_0, x_0, \frac{t}{2}) \leq C k_P^{M}(x_0, x_0, t)\qquad  \mbox{for all }  \ t > T.
\end{equation}
Then
\begin{enumerate}
  \item There exists $\vge_0 >0$ such that $k_{P-\gl}^{M}\asymp k_{P_{\vge}-\gl}^{M}$ for all $|\vge| <\vge_0$ and all $\gl\leq 0$.


  \item Suppose further that $V\geq 0$, then  $k_{P-\gl}^{M}\asymp k_{P_{\vge}-\gl}^{M}$ for all $-\infty <\vge <\vge_0$ and all  $\gl\leq 0$.


  \item Suppose further that $P - V$ is subcritical in $M$ and satisfies \eqref{assumption1}. Then $\kP \asymp k_{P - V}^{M}$.


  \item  Assume that $P$ is a symmetric subcritical operator with locally regular coefficients, and that \eqref{assumption1} is satisfied. Then assertions (1) - (3) hold true (without the $C^\infty$-assumption, and the assumption on the metric) provided $V$ is a bounded measurable potential with compact support.
  \end{enumerate}
\end{thm}
The following theorem extends Theorem~\ref{main} from the class of compactly supported potentials to the class of {\em small perturbations} (see Definition~\ref{def_sp}).
\begin{thm}\label{spmain}
Suppose that the Riemannian manifold $(M, \mathfrak{g})$, the operator $P$, and its kernel $\kP$ satisfy the assumptions of Theorem~\ref{main}. Let $V \in L^{p}_\loc(M)$ be a small perturbation with respect to $\kP$ in $M$, where $p > {N}/{2}$.
\begin{enumerate}
\item Then there exists $\vge_0 >0$ such that $k_{P-\gl}^{M}\asymp k_{P_{\vge}-\gl}^{M}$ for all $|\vge| <\vge_0$ and all $\gl\leq 0$.


\item  Suppose further that $V\geq 0$, then  $k_{P-\gl}^{M}\asymp k_{P_{\vge}-\gl}^{M}$ for all $-\infty <\vge <\vge_0$ and all  $\gl\leq 0$.


\item Suppose further that $P - V$ is subcritical in $M$ and satisfies the doubling condition \eqref{assumption1} (without any sign assumption on $V$), then $\kP \asymp k_{P - V}^{M}.$
\end{enumerate}
Moreover, if $V$ is only a \emph{semismall perturbation}, then $(1)$ and $(2)$ hold true with the semi-equivalence replacing the equivalence assertion.
\end{thm}
\begin{rem}\em{
Assumption \eqref{assumption1} necessarily implies that  $\lambda_{0}(P, \mathbf{1}, M) = 0.$
 Indeed, if  $\lambda_{0}>0,$ then \eqref{large-eigen} implies that $\kP$ decays exponentially as $t\to\infty$, and this contradicts \eqref{assumption1}.

On the other hand, if $P\geq 0$ in $M$, and $k_P^{M}\asymp k_{P_{\vge}}^{M}$ for all $|\vge| <\vge_0$, then  \eqref{eq_hkp-gl} implies that  $k_{P-\gl}^{M}\asymp k_{P_{\vge}-\gl}^{M}$ for all $\gl \leq \gl_0$ and $|\vge| <\vge_0$ (and in particular, $P_\vge-\gl_0$ is subcritical in $M$, see Proposition~\ref{thm-critical} below).
}
\end{rem}
\begin{rem}\em{
If $\gl_0>0$ and $P-\gl_0$ satisfies the assumptions of Theorem~\ref{main} or Theorem~\ref{spmain}, then the conclusions of these theorems hold true for $P-\gl$ for all $\gl\leq \gl_0$ (see e.g. Example~\ref{ex5}).
}
\end{rem}
\begin{rem}\label{rem_Grig}
{\em
The doubling condition \eqref{assumption1} is not very restrictive. Clearly, the positive minimal heat kernels of the Laplacian on $\R^N$ with $N\geq 3$, and  on the upper half-space $\R_+^N$ with $N\geq 1$ satisfy \eqref{assumption1} (see, \cite{SC1,SC}). In Subsection~\ref{subsec_ex} we provide further examples of manifolds $M$ and operators $P$ satisfying \eqref{assumption1}.


 On the other hand, A.~Grigor'yan kindly pointed out to us that for some model subcritical manifolds $M$ with $\gl_0=0$ and with exponential volume growth $V(r)=\exp(r^\ga)$, where $0<\ga<1$,  the heat kernel satisfies the on-diagonal estimates $$k_{-\Gd}^M(x_0,x_0,t)\asymp \exp(-ct^{\ga/(2-\ga)}).$$
So, the doubling condition $\eqref{assumption1}$ is not satisfied (see, Example~5.36 and Theorem~5.42 in \cite{greg0}).
 }
\end{rem}
In the critical case we have the following result.
\begin{proposition}\label{thm-critical}
Assume that $P$ is critical in $M$, and let $V$ be a nonzero potential.

Then for any $\gl \leq 0$ there does not exist any $\varepsilon_{0} > 0$
such that  $k_{P- \lambda}^{M} \asymp k_{P_{\varepsilon} - \lambda}^{M}$   for all   $|\varepsilon|< \varepsilon_{0}$. Moreover, the corresponding heat kernel $\kP$ does not satisfy
the $3k$-inequality \eqref{3k} with respect to any nonzero potential $V$.
 \end{proposition}
 \begin{proof}
 It follow from \cite[Theorem~3.1]{YP1}, that if $P$ is critical, then there exists at most one $\vge_1\neq 0$ such that $P_\vge$ is also critical in $M$. Hence, $\kP\not \asymp k_{P_\vge}^M$ for all $\vge\neq\vge_{1}$.  In light of \eqref{eq_hkp-gl}, we conclude the result for all $\gl\leq 0$.  The last part of the proposition follows from  the proof of first part and Theorem~\ref{thm_3k_implies_eq}.
\end{proof}
\begin{rem}{\em
Proposition~\ref{thm-critical} is counter intuitive, since in the context of the Green function, even if $P-\gl_0$ is critical in $M$, yet for any nonzero potential $V$ with a compact support, and any $\gl<\gl_0$  there exists $\vge_0=\vge_0(V,\gl)>0$ such that $G_{P-\gl}^{M} \asymp G_{P_{\varepsilon}-\lambda}^{M}$ for any $|\varepsilon| < \varepsilon_{0}$.
}
\end{rem}
\begin{rem}{\em
Let $P$ be a subcritical operator in $M$ and $V$ a nonzero potential. Then $P-\lambda_{0}$ is subcritical if there exists $\varepsilon_{0}$ such that $k^{M}_{P-\lambda} \asymp k^{M}_{P_{\varepsilon} - \lambda}$ for all $|\varepsilon| < \varepsilon_{0}$ for  some $\lambda \leq \lambda_{0}$.
}
\end{rem}
The proof of Theorem~\ref{main} relies on a suitable $3k$-inequality (see Definition~\ref{def3k} below). We note that an analogous $3G$-inequality is used frequently  for proving the equivalence of Green functions (see for example \cite{MM1,YP1,YP2}).
\section{Preparatory results}\label{secpreliminaries}
In the present section we recall some basic properties of the heat kernel, define the notion of $3k$-inequality, and prove some basic general results concerning the equivalence of heat kernels. The lemma below summarizes some fundamental properties of the heat kernel.
\begin{lem}\label{properties}
Let $P$ be an elliptic operator either of the form \eqref{P} or \eqref{operator}, which is nonnegative in $M$. Then the positive minimal heat kernel $k_{P}^{M} (x, y, t)$ satisfies the following properties:
\begin{enumerate}
\item $k_{P}^{M} (x, y, t)$  satisfies the Chapman-Kolmogorov equation (the semigroup property)
\[
k_{P}^{M} (x, y, s + t) = \int_{M} k_{P}^{M} (x, z, s) k_{P}^{M} (z, y, t) \dz \quad  \forall s, t > 0 \mbox{ and } \forall \ x, y \in M.
\]

\item
$
  \kP(x, y, t) \geq 0, \quad  k_{P^\star}^{M}(x, y, t)  = \kP(y, x, t) \quad \forall t > 0  \mbox{ and } \forall \ x, y \in M.
$
\item The heat kernel is monotone increasing as a function of the domain.


\item If $V\geq 0$, then $k_{P+V}^{M}\leq \kP$


\hspace{-2cm}  Suppose further that $P$ is symmetric. Then:


\item   $\kP(x, y, t)   \leq  \kP(x, x, t)^{\frac{1}{2}} \kP(y, y, t)^{\frac{1}{2}}\qquad \forall \ t > 0 \mbox{ and } \forall \  x, y \in M.$


\item The function $t \mapsto \kP(x, x, t)$ is positive, monotone decreasing and log-convex for all $x \in M$.


\item Assume that  $P$ is a nonnegative selfadjoint operator on $L^{2}(M, \dx)$, then
\[
\mathrm{e}^{-Pt}f(x) = \int_{M} k_{P}^{M}(x,y,t) f(y) \dy
\]
for all $ t > 0$ and $f \in L^{2}(M,\dx).$
\end{enumerate}
\end{lem}
For the proof of the above lemma we refer to \cite[Lemma~1]{DA1}.


In the sequel we need the following log-convexity property of the heat kernels with respect to a perturbation by a potential $W$ (see for example \cite[Lemma B.7.73]{S82}).
\begin{prop}\label{pro_conv}
Suppose that the elliptic operators $P_0$ and $P_1:=P_0+W$ both admit positive minimal heat kernels $k_0$ and $k_1$, respectively, in $M$. Then
for any $0\leq \ga \leq 1$, the operator $P_\ga:= (1-\ga) P_0 +  \ga P_1$ admits a positive minimal heat kernel $k_\ga$ in $M$,
and $k_\ga$ satisfies
\begin{equation}\label{convexity}
k_\alpha(x,y,t) \leq (k_{0}(x,y,t))^{(1- \alpha)} (k_{1}(x,y,t))^{\alpha}\qquad \forall \ x, y \in M, \mbox{ and } t > 0.
\end{equation}
\end{prop}
\begin{defi}\label{def3k}{\em
Let $P$ be a subcritical operator defined on $M.$ We say that the heat kernel $\kP$ satisfies the {\em $3k$-inequality} with respect to $V$ if there exists a constant $C > 0$ such that the following inequality holds true:
\begin{equation}\label{3k}
\!\int_{0}^{t}\!\! \int_{M}\!\!\! k_{P}^{M}(x,z, t-s) |V(z)| k_{P}^{M}(z,y,s)\! \dz \!\ds\!
   \leq \! C k_{P}^{M}(x,y,t) \,\,\; \forall x, y \!\in\! M, \mbox{ \!and } t \! > \!0.
\end{equation}
 We say that the heat kernel $\kP$ satisfies the  {\em restricted $3k$-inequality} with respect to $V$ if for any $T>0$ there exists a constant $C(T) > 0$ such that the following inequality holds true:
\begin{equation}\label{3kT}
\int_{0}^{t} \int_{M} k_{P}^{M}(x,z, t-s) |V(z)| k_{P}^{M}(z,y,s) \dz \ds \leq  C(T) k_{P}^{M}(x,y,t)
\end{equation}
for all $x, y \in M$  and $0 < t \leq T$.
 }
\end{defi}
If $V$ is a bounded potential, then the heat kernel satisfies the restricted $3k$-inequality. Indeed, the Chapman-Kolmogorov equation clearly implies:
\begin{prop}\label{thm3kbounded}
Let $P$ be an elliptic operator either of the form \eqref{P} or \eqref{operator}, which is nonnegative in $M$, and let $V$ be a bounded potential. Then the following
restricted $3k$-inequality holds true
\begin{equation*}
\!\int_{0}^{t}\!\! \int_{M}\!\!\! k_{P}^{M}(x, z, t-s) |V(z)| k_{P}^{M}(z, y, s)\! \dz \ds\! \leq \! T\|V\|_\infty k_{P}^{M}(x, y, t)
\end{equation*}
for all $x, y \in M$ and  $0 < t \leq T$.
\end{prop}
The next theorem asserts that if $\kP$ satisfies the $3k$-inequality, then for small $|\vge|$, we have $k_{P_\varepsilon}^{M}\asymp \kP$ (cf. \cite[Theorem 5.3]{FKP}).
\begin{thm}\label{thm_3k_implies_eq}
\emph{(1)} Let $V$ be a potential such that  $\kP$ satisfies the $3k$-inequality \eqref{3k}. Then there exists $\vge_0 >0$ such that $k_{P_\varepsilon}^{M}\asymp \kP$ for all $|\vge| <\vge_0$.


\emph{(2)} If $\kP$ satisfies the restricted $3k$-inequality \eqref{3kT}, then
for any $T>0$ there exits positive $\vge_0(T)$ such that for all $\vge<\vge(T)$
\begin{equation}\label{eq_eqiv4}
k_{P_\vge}^M \asymp
k_{P}^M \qquad   \mbox{ on } M \times M \times (0,T].
\end{equation}

\emph{(3)} Under the assumptions of either \emph{(1)} or  \emph{(2)}, let $\vge$ be such that \eqref{eq_eqiv4} holds true with  $0<T\leq \infty$. Then the heat kernel $k_{P_\varepsilon}^{M}$ satisfies the resolvent equations
\begin{multline}\label{r_eq3}
   k_{P_\varepsilon}^{M}(x, y ,t)  = \kP(x, y, t) \!+\! \varepsilon \!\int_{0}^{t} \!\!\int_{M} \!\!\kP(x, z, t\!-\!s) V(z) k_{P_\vge}^{M} (z, y, s) \dz\! \ds\\[2mm]
=\kP(x, y, t)\! +\! \varepsilon \int_{0}^{t}\!\! \int_{M} \!\!k_{P_\vge}^{M}(x, z, t\!-\!s) V(z) \kP (z, y, s) \dz \!\ds
\end{multline}
for all $(x,y,t)\in M \times M \times (0,T)$.
\end{thm}
\begin{proof}[Proof of \emph{(1)} and \emph{(2)}]
Fix $0<T\leq \infty$ and $y\in M$. Consider the iterated kernel
\begin{equation*}\label{i2}
k^{(i)}_{P}(x,y,t) := \left\{
                       \begin{array}{ll}
                         \kP(x,y,t) & i=0, \\[2mm]
                         \int_{0}^{t} \int_{M} k^{(i-1)}_{P}(x, z, t-s) V(z) \kP(z, y, s) \dz \ds & i\geq 1.
                       \end{array}
                     \right.
\end{equation*}
It follows from  the $3k$-inequality \eqref{3k} (or the restricted $3k$-inequality \eqref{3kT}) that for all $0<t<T$ we have
\begin{equation}\label{eq_ki}
k^{(i)}_{P}(x,y,t)\leq C^i \kP(x,y,t).
\end{equation}
Hence,
\begin{equation}\label{eq_sum}
\sum_{i = 0}^{\infty} |\varepsilon|^{i} |k^{(i)}_{P}(x,y,t)| \leq
  \frac{1}{1-C|\varepsilon|} \kP(x, y, t),
\end{equation}
provided $|\varepsilon| < C^{-1}$.

Fix such $\vge$. Using a standard parabolic regularity argument, it follows that the Neumann series
\[
H^{\vge}_{P}(x, y, t) := \sum_{i = 0}^{\infty} \varepsilon^{i} k^{(i)}_{P}(x, y, t)
\]
converges locally uniformly in $M\times (0,T)$ to a positive fundamental solution of the equation
$(u_t+P_\vge) u = 0$. Hence, $k_{P_\varepsilon}^{M}(x, y, t)$ exists, and by the minimality of the heat kernel and \eqref{eq_sum} we obtain
\begin{equation}\label{eq_est}
    k_{P_\varepsilon}^{M}(x, y, t)\leq H^{\vge}_{P}(x, y, t)\leq  \frac{1}{1-C|\varepsilon|} \kP(x, y, t).
\end{equation}

Let $M_j$ be an exhaustion of $M$, i.e., a sequence of smooth,
relatively compact subdomains of $M$ such that $y\in M_1$,
$M_{j}\Subset M_{j+1}$ and
$\cup_{j=1}^{\infty}M_{j}=M$.

Using the resolvent equation (Duhamel's principle) in  $M_j$
\begin{equation}\label{Duham}
 k_{P_\varepsilon}^{M_j}(x, y ,t)  = k_{P}^{M_j}(x, y, t) + \varepsilon \int_{0}^{t} \int_{M_j} k_{P}^{M_j}(x, z, t-s) V(z) k_{P_\varepsilon}^{M_j} (z, y, s) \dz \ds,
\end{equation}
and by the dominated convergence theorem,  we obtain that $k_{P_\varepsilon}^{M}$ satisfies the resolvent equation
$$k_{P_\varepsilon}^{M}(x, y ,t)  = \kP(x, y, t) + \varepsilon \int_{0}^{t} \int_{M} \kP(x, z, t-s) V(z) k_{P_\vge}^{M} (z, y, s) \dz \ds.
$$
Moreover, by the resolvent equation and inequality \eqref{3k}, we have
\begin{multline*}
k_{P_\varepsilon}^{M}(x, y, t)  = \kP(x, y, t) + \varepsilon \int_{0}^{t} \int_{M} \kP(x, z, t-s) V(z) k_{P_\varepsilon}^{M} (z, y, s) \dz \ds \\[2mm]
 \geq \kP(x, y, t) - \frac{|\varepsilon|}{1 - C|\varepsilon|} \int_{0}^{t} \int_{M} \kP(x, z, t-s) | V(z) | \kP(z, y, s) \dz \ds \\\\[2mm]
 \geq \kP(x, y, t) - \left( \frac{C|\varepsilon|}{1 - C|\varepsilon|} \right) \kP(x, y, t) = \left( \frac{1 -2 C|\varepsilon|}{1 - C|\varepsilon|} \right) \kP(x, y, t).
\end{multline*}
Hence, for $|\vge| < 1/(2C)$ we have $k_{P_\varepsilon}^{M}\asymp \kP$, which in turn implies that
$H^{\vge}_{P} \asymp k_{P_\varepsilon}^{M}$. The minimality of $k_{P_\varepsilon}^{M}$ implies now that
$$k_{P_\varepsilon}^{M}(x, y, t)= H^{\vge}_{P}(x, y, t) := \sum_{i = 0}^{\infty} \varepsilon^{i} k^{(i)}_{P}(x, y, t)$$ for any $|\vge| < 1/(2C)$.
This proves  parts (1) and (2) of the theorem.

Part~\mbox{(3)} of the theorem follows from the resolvent equation \eqref{Duham} in $M_j$, the $3k$-inequality, and the dominated convergence theorem.
\end{proof}
\begin{rem}{\em
It is evident from  Theorem~\ref{thm_3k_implies_eq} that in order to prove theorems~\ref{main} and \ref{spmain}, it is enough to establish the $3k$-inequality~\eqref{3k}.
 }
\end{rem}
For a perturbation by a nonnegative potential $V$, we have
\begin{lem}[{\cite[Corollary 2]{FKP}}]\label{lem1}
Let $P$ be a subcritical operator, and let $V$ be a nonnegative potential. Suppose that $k_{P_{\vge_0}}^{M}\asymp k_{P}^{M}$ for some $\vge_0 >0$. Then
$k_{P_\vge}^{M}\asymp k_{P}^{M}$ for any $\varepsilon <\vge_0$.
\end{lem}
We provide here a detailed proof of the above lemma.
\begin{proof}
By the generalized maximum principle, if $\vge_1 < \vge_2$, then
\begin{equation}\label{mono}
k_{P_{\vge_1}}^{M} \leq k_{P_{\vge_2}}^{M}.
\end{equation}
So, by our assumption, there exists $C>0$ such that
$$\kP\leq k_{P_{\vge_0}}^{M} \leq C\kP.$$

Let $0\leq \vge\leq \vge_0.$ Then, by \eqref{mono} and Proposition~\ref{pro_conv} with  $0\leq \ga:= \frac{\vge}{\vge_0}\leq 1$, we have
\begin{equation*}\label{convexity1}
k_{P_\varepsilon}^{M}\leq (\kP)^{1- \alpha} (k_{P_{\vge_0}}^{M})^{ \alpha}\leq C^\ga\kP=C^{\vge/{\vge_0}}\kP.
\end{equation*}
On the other hand,  if $\vge< 0$, then  by \eqref{mono}  and Proposition~\ref{pro_conv}
we have with $\ga=-\vge/(\vge_0-\vge)$
\[
k_{P_{\vge}}^{M} \leq k_{P}^{M}   \leq (k_{P_{\vge_0}}^{M})^{ \alpha} (k_{P_{\vge}}^{M})^{1- \alpha}\leq C^\ga (k_{P}^{M})^{ \alpha}(k_{P_{\vge}}^{M})^{1- \alpha} ,
\]
and hence,
$$k_{P_{\vge}}^{M} \leq k_{P}^{M}   \leq C^{\ga/(1-\ga)}k_{P_{\vge}}^{M}=C^{-\vge/{\vge_0}}k_{P_{\vge}}^{M}.$$
So,
\begin{equation*}\label{convexity2}
C^{\vge/{\vge_0}}\kP\leq  k_{P_\varepsilon}^{M}\leq \kP,
\end{equation*}
and this completes the proof of the lemma.
\end{proof}
\begin{remark}\label{remBZ}
{\em
In \cite{BDS,BS,Z2}, the authors consider the special case of the Laplacian on $\R^N$ and {\em signed} potential perturbations. It is proved there that for $V\geq  0$ which is in a certain $L^p$ subspaces, $k_{-\Gd}^{\R^N}\asymp k_{-\Gd-\vge V}^{\R^N}$ for {\em any}  $\vge \leq 0$.
Our Lemma~\ref{lem1} and Theorem~\ref{spmain}, applied to this particular case, extend these results even for signed potentials $V$, since in this case, by our results, the interval of equivalence is $(-\infty,\vge_0)$, where $\vge_0>0$ (and not only $\R_-$).
}
\end{remark}
Recall that by \cite{YP1}, the set
$$ S=S(P,V,M):=\{\vge\in \R \mid P_\vge \geq 0\mbox{ in } M\}$$
is a closed convex set, which contains the convex set
$$S_+=S_+(P,V,M):=\{\vge\in \R \mid P_\vge  \mbox{ is subcritical} \},$$
and
$$\{\vge\in \R \mid P_\vge  \mbox{ is critical} \}\subset \partial S. $$
Moreover,  if $V$ is a small perturbation of $P$ in the sense of Green functions, then $S_+=\mathrm{int}\,S$, and $G_{P_\vge}^M\asymp G_{P}^M$ for any $\vge\in S_+$ (see \cite{YP1}).


We note that, in general, the convexity of the set $$\{(\gl,\vge)\in \R^2\mid \gl\leq \gl_0,\;\; \vge\in S_+(P-\gl ,V,M)\},$$
implies that for any $\gl\leq\gl_0$, we have
\begin{equation}\label{eqS_+eq}
\{\vge\in \R \mid k_{P_\vge-\lambda }^M\asymp k_{P-\lambda }^M\} \subset S_+(P-\gl_0,V,M).
\end{equation}

The following lemma shows that under some conditions we have
\begin{equation}\label{eqS_+}
   \{\vge\in \R \mid k_{P_\vge}^M\asymp \kP\}= S_+(P-\gl_0,V,M).
\end{equation}
\begin{lem}\label{global-eq}
Let $P$ and $P-V$ be two subcritical elliptic operators  such that for some $0 < \alpha < \beta < 1,$ there holds
$k^{M}_{P_{\alpha}}\asymp \kP $  and $k^{M}_{P_{\beta}}\asymp k^{M}_{P-V}$. Then
\[
\kP \asymp k^{M}_{P-V}.
\]
\end{lem}
\begin{proof}
By \eqref{eqS_+eq}, we may assume that $\gl_0(P,\mathbf{1},M)=0$.  Proposition \ref{pro_conv} and the lemma's hypothesis $\kP \asymp k^{M}_{P_{\alpha}}$ imply that
\begin{equation*}\label{glo1}
\kP \asymp k^{M}_{P_{\alpha}} \leq (\kP)^{1- \alpha} (k^{M}_{P-V})^{\alpha}.
\end{equation*}
This implies $ C_1 \kP\leq k^{M}_{P-V}$. Similarly,
\begin{equation*}\label{glo2}
 k^{M}_{P-V} \asymp k^{M}_{P_{\beta}} \leq (\kP)^{1- \beta} (k^{M}_{P-V})^{\beta},
\end{equation*}
implies $ k^{M}_{P-V}\leq C_2 \kP$. Hence, the lemma is proved.
\end{proof}
In the study of equivalence of heat kernels, one would expect that as in the elliptic case (see for example \cite{YP1,YP2}), the local Harnack inequality should play a pivotal role. Unfortunately, the parabolic Harnack inequality for nonnegative solutions is weaker than the elliptic one. Nevertheless, in the {\em symmetric case}, the heat kernel satisfies the following elliptic-type Harnack inequality due to E.~B.~Davies \cite[Theorem~10]{DA1}.
\begin{lem}[Davies-Harnack inequality for the heat kernel]\label{harnack}
Fix a compact subset $A$ of $M$, and $T > 0$. Then there exists a positive constant $C:= C(T,  A,P)$ such that
\begin{equation}\label{eqharnack}
\sup_{x,y\in A}\kP(x, y, t) \leq  C \inf_{x,y\in A}\kP(x, y, t)    \qquad \forall \ t \geq T.
\end{equation}
\end{lem}
\section{Proof of Theorem~\ref{main}}\label{secproof}
The proof of Theorem~\ref{main} hinges on the following key proposition.
\begin{prop}\label{mainprop}
Assume that $P$, $V$ and $\kP$ satisfy the assumptions of the first part of Theorem~\ref{main}. In addition, assume that the diameter of $\supp V$ is small enough. Then the corresponding heat kernel $\kP$ satisfies the $3k$-inequality \eqref{3k}. Consequently, there exists $\vge_0>0$ such that $k_{P_\vge}^M\asymp \kP$ for all $|\vge|<\vge_0$.
 \end{prop}
\subsection{Short time asymptotic}
One of the key steps of the proof of the $3k$-inequality of Proposition~\ref{mainprop} relies on  the local short time asymptotic of the heat kernel $k_{P}^{M}(x, y, t).$
Recall that two-sided short time {\em estimates} of the heat kernel have been extensively studied in the past forty years. However, for our purpose, we need the local short time {\em asymptotic} of the heat kernel which is given
by the following theorem of Y.~Kannai \cite{YK} (see also \cite{MP} for the result in the compact case). For the global analogues result see Section~4 of \cite[Theorem 4.1]{YK}, and for subsequent developments of the these results see \cite{BER,CHA,TIN}.
\begin{lemma}[\cite{YK}]\label{shorttime}
Assume that an elliptic operator $P$ either of the form \eqref{P} or \eqref{operator} with $C^\infty$-coefficients is defined on a smooth noncompact manifold $M$. Let $d(x,y)$ be the Riemannian distance induced by the principal part of the operator $P$.

For any relatively compact set $K \subset M \times M,$ there is a $\delta > 0$
and  smooth functions $H_{n}(x, y)$, $n=0,1,\ldots$, defined on $K$ such that the following asymptotic expansion
\begin{equation}\label{eqshort}
\kP(x, y, t) \sim \left( \frac{1}{4 \pi t} \right)^{\frac{N}{2}} \mathrm{e}^{\frac{- d(x, y)^2}{4t}} \sum_{n = 0}^{\infty} H_{n}(x, y) t^{n}
\end{equation}
holds locally uniformly as $t \rightarrow 0$ in $K$, whenever $d(x, y) < \delta$. Moreover,
\[
H_{0}(x, y) > 0 \ \ \mbox{and} \ H_{0}(x, x) = 1.
\]
In particular, for small enough $t>0$, and for $x, y$ in a small compact set in $M \times M$,  we have
\[
\kP(x, y, t) \asymp \left( \frac{1}{4 \pi t} \right)^{\frac{N}{2}} \mathrm{e}^{\frac{- d(x, y)^2}{4t}}.
\]
\end{lemma}
Next, we state and prove another key ingredient for the proof of the $3k$-inequality.
\begin{lem}\label{maximum}
Let $V \in L^{p}(M), p > \frac{N}{2}$ be a potential with compact support $K$, and let $A$ be a bounded domain with a smooth boundary such that $K\Subset A$. Assume that there exists a constant $C > 0$ such that
\begin{equation}\label{eqmaximum}
\int_{0}^{t} \int_{K} \kP(x, z, t-s) |V(z)| \kP(z, y, s) \dz \ds \leq C \kP(x, y, t),
\end{equation}
for any  $x, y \in A$, and $ t > 0$. Then
\begin{equation}\label{eqmaximum1}
\int_{0}^{t} \int_{K} \kP(x, z, t-s) |V(z)| \kP(z, y, s) \dz \ds \leq C \kP(x, y, t),
\end{equation}
for any $x, y \in M,$ and $ t > 0.$
\end{lem}
\begin{proof}
Following common practice, in the sequel, the letter $C$ will denote an irrelevant positive constant, the value of which might change from line to line, and even in the same line.

Fix $y \in A$, and define
\begin{equation} \label{solution1}
U_{y}(x, t): = \int_{0}^{t} \int_{K} \kP(x, z, t-s) |V(z)| \kP(z, y, s) \dz \ds .
\end{equation}
By \eqref{eqmaximum},
$$U_{y}(x,t) \leq C \kP(x, y, t)\qquad  \forall  x\in \partial A \mbox{ and } t > 0.$$
Moreover,  $U_{y}$ is a solution of the equation
$$
\frac{\partial }{\partial t} U_{y} + P U_{y} =
  |V(x)| \kP(x,y, t) \qquad  x\in M \mbox{ and }  t>0.
$$
In particular, $\frac{\partial }{\partial t} U_{y} + P U_{y} =0$ for all $x$ outside $K$ and $t>0$.


Let $\{ M_{n} \}_{n=0}^\infty$ be an exhaustion of $M$ such that $A\subset M_{0}$, and set
\[
 U_{y,n}(x, t) := \int_{0}^{t} \int_{A} k_{P}^{M_{n}}(x, z, t-s) |V(z)| k_{P}^{M_{n}}(z, y, s) \dz \ds,
 \]
 where $k_{P}^{M_{n}}(x, y, t) $ is the Dirichlet  heat kernel of $P$ on  $M_{n}$.


Recall that as a function of $x$, the heat kernel  $k_{P}^{M_{n}}(x, y, t)$ satisfies the equation $ \frac{\partial }{\partial t} k_{P}^{M_{n}} + P k_{P}^{M_{n}} = 0 $ in $M_n\times (0,\infty)$. Moreover, since $U_{y,n}$ and $k_{P}^{M_{n}}$ converges locally uniformly to $U_{y}$ and $k_{P}^{M}$, respectively,  it follows that for any  $\vge>0$ there is $N_\vge$, such that for any $n\geq N_\vge$
 $$U_{y,n}(x) \leq (C+\vge) k_{P}^{M_{n}}(x, y, t)\qquad  \forall  x\in \partial A \mbox{ and } t>0.$$
Therefore, for such $n$, we have
  $$
\left \{\begin{array}{lllll}
\frac{\partial}{\partial t} U_{y,n} + P U_{y,n} = 0 \quad &\mbox{ in }   (M_{n} \setminus A)\times (0,\infty),\\[2mm]
U_{y,n} \leq (C+\vge) k_{P}^{M_{n}}  &\mbox{ on }   \partial A \times (0,\infty), \\[2mm]
U_{y,n} = 0 & \mbox{ on }  \partial M_{n}\times (0,\infty),\\[2mm]
U_{y,n} = 0 & \mbox{ on }   (M_{n}\setminus A)\times \{0\}.
 \end{array}\right.$$
The generalized maximum principle implies that
 \[
 U_{y,n} \leq (C+\vge) k_{P}^{M_{n}} \ \ \mbox{ on } (M_{n} \setminus A)\times (0,\infty),
 \]
Letting $n\to \infty$ we arrive at
 \begin{equation}\label{eq3k6}
 U_{y} (x, t) \leq C \kP(x, y, t) \qquad \forall x \in M,  y \in A \mbox{ and } t > 0.
 \end{equation}
 Next,  we fix $x \in M$ and define for $y \in M$
  \[
 U_{x}^\star(y, t): = \int_{0}^{t} \int_{A} \kP(x, z, t-s) |V(z)| \kP(z, y, s) \dz \ds.
 \]
Then as a function of $y$, $U_{x}^\star$ is a solution of the equation
$$
\frac{\partial }{\partial t} U_{x}^\star + P^\star U_{x}^\star =
  |V(y)| \kP(x, y, t) \qquad  y \in M \mbox{ and }  t > 0.
$$
  In particular, $\frac{\partial }{\partial t} U_{x}^\star + P^\star U_{x}^\star =0$ for all $y$ outside $A$.


Since $U_{x}^\star(y, t)= U_{y}(x, t)$, estimate \eqref{eq3k6} implies
 $$U_{x}^\star(y, t)  \leq C k_{P}^{M}(x, y, t) \qquad \forall y \in  A \mbox{ and } x \in M.$$
Hence, the above exhaustion and comparison arguments finally imply
 \begin{equation}\label{eqmaximum3}
\int_{0}^{t} \int_{A} \kP(x, z, t-s) |V(z)| \kP(z, y, s) \dz \ds \leq C \kP(x, y, t),
\end{equation}
for any $x, y \in M,$ and $ t > 0.$
\end{proof}
Having proven Lemma~\ref{maximum}, we turn to the proof the $3k$-inequality.
\begin{proof}[Proof of Proposition~\ref{mainprop}] By Lemma~\ref{maximum}, it is sufficient to prove the $3k$-inequality for
for all $x,y\in A$  and all $t > 0 $. So, it is enough to prove the existence of a constant $C>0$ such that
\begin{equation}\label{SV}
S(V, x, y, t) := \int_{0}^{t} \int_{M} \frac{\kP(x, z, t-s) \kP(z, y, s)}{\kP(x, y, t)} |V(z)| \dz \ds \leq C
\end{equation}
for all $x, y \in A$  and all $t > 0$.


The proof is divided into several steps.  We fix an arbitrary small $\gd_{0} > 0$ (to be chosen later),  and prove the boundedness of $S(V, x, y, t)$ in two separate regions;
$t \geq \gd_{0}$ and $0 < t < \gd_{0}$.

\medskip

{\bf{Step 1}:} In this step we estimate \eqref{SV} when $ t \geq \gd_{0} $ and $x, y  \in A,$ where $A$ is smooth compact subset of M containing $K=\mbox{supp}\, V$.  Fix  $0 < \delta < \frac{\gd_{0}}{2}$.  Fubini's theorem yields,
\begin{align}\label{I1I2}
\nonumber & \int_{0}^{t} \int_{M} \kP(x, z, t-s) \kP(z, y, s) |V(z)| \dz \ds  \\ \nonumber & = \int_{A} \left( \int_{0}^{\delta}  \kP(x, z, t-s) \kP(z, y, s)  \ds \right)  |V(z)| \dz \\
& +  \int_{A} \left( \int_{\delta}^{t} \kP(x, z, t-s) \kP(z, y, s)  \ds \right)   |V(z)| \dz.
\end{align}
 Consider the first term of \eqref{I1I2}, namely,
$$I _{1}^{\delta }   := \int_{A} \left(  \int_{0}^{\delta}    \kP(x, z, t-s) \kP(z, y, s)  \ds  \right)  |V(z)| \dz.$$
Since $t>\gd_0$, we have for $0<s<\gd$
$$\gd< \frac{t}{2}< t - \delta < t - s <t.$$
Hence, in light of parts (5) and (6) of Lemma \ref{properties}, and Davies-Harnack inequality \eqref{eqharnack},  we obtain
\begin{align*}
I _{1}^{\delta } & \!\leq \! \int_{A}\!\! \left( \!\int_{0}^{\delta}
\!\! (\kP(x, x, t-s))^{\frac{1}{2}}  (\kP(z, z, t-s))^{\frac{1}{2}} \kP(z, y, s) \!\ds \!\!  \right)  |V(z)| \!\dz \\
 & \leq C \left(\kP(x, x, \frac{t}{2}) \right)^{\frac{1}{2}} \left( \kP(y, y, \frac{t}{2}) \right)^{\frac{1}{2}}  \int_{A}  \left( \int_{0}^{\infty} \kP(z, y, s) \ds \right) |V(z)| \dz.
\end{align*}
Using our assumption that $P$ is subcritical in $M$, the Davies-Harnack inequality \eqref{eqharnack}, and the doubling condition \eqref{assumption1}, we get
\begin{align}\label{impstep}
I _{1}^{\delta } \leq  C \left( (\kP(x_0, x_0, t) \right)^{\frac{1}{2}} \left( \kP(x_0, x_0, t) \right)^{\frac{1}{2}}  \int_{A} G_P^M(z, y) |V(z)| \dz,
\end{align}
where $G_P^M$ is the Green function of the operator $P$ in $M$. Consequently,  the Davies-Harnack inequality \eqref{eqharnack} for the heat kernel implies
\begin{align}\label{impstepn}
I _{1}^{\delta } \leq  C(\gd,A) \kP(x, y, t)   \int_{A} G_P^M(z, y) |V(z)| \dz \qquad \forall x, y \in A, \; t>\gd.
\end{align}
On the other hand, the well known behaviour of the Green function near singularity, and the local elliptic Harnack inequality imply that there exist a positive constant $C$ such that
 \[
C^{-1} |z-y|^{2- N} \leq G(z,y) \leq C|z-y|^{2 - N} \qquad \forall z,y\in A.
\]
Hence, the H\"older inequality with $p > N/2$ and $p'=p/(p-1)$ yields
\begin{multline}\label{G_est}
\int_{A} G(z, y) |V(z)| \dz
\leq C \left( \int_A |y -z|^{(2-N){p'}} \dz \right)^{\frac{1}{p'}} \left( \int_{A} |V|^{p} \dz \right)^{\frac{1}{p}} \\[2mm]  \leq C(K,p,N)\|V\|_p \qquad \forall  y\in A.
\end{multline}
 Hence, by substituting \eqref{G_est} into \eqref{impstepn} we obtain
\begin{equation}\label{impstep1}
I _{1}^{\delta } \leq  C  \kP(x, y, t) \qquad \forall x, y \in A \mbox{ and } t \geq \gd_{0}.
\end{equation}
where the constant $C$ depends on $\delta, A, p, N, \|V\|_p $.

Next, consider the second term of \eqref{I1I2}, namely,
\begin{equation*}
I_{2}^{\delta} := \int_{A} \left( \int_{\delta}^{t}    \kP(x, z, t-s) \kP(z, y, s)  \ds   \right) |V(z)| \dz.
\end{equation*}
Acting as for $I_{1}^{\delta}$ we obtain
 \begin{align*}
 I_{2}^{\delta}  & \!\leq \!\!\int_{A} \!\left( \! \int_{\delta}^{t/2}  \!\!\kP(x, x, t-s) ^{\frac{1}{2}}   \kP(z, z, t-s)^{\frac{1}{2}}  \kP(z, y, s)  \ds \right) |V(z)| \dz \\+
& \int_{A} \left( \int_{t/2}^{t} \left( \kP(z, z, s) \right)^{\frac{1}{2}}  \left( \kP(y, y, s) \right)^{\frac{1}{2}}  \kP(x, z, t-s)  \ds  \right) |V(z)| \dz \\
 & \leq C \left( \kP(x, x, \frac{t}{2}) \right)^{\frac{1}{2}} \left(\kP(y, y, \frac{t}{2}) \right)^{\frac{1}{2}} \int_{A} G(z, y) |V(z)| \dz \\
 &+
  C\left( \kP(y, y, \frac{t}{2}) \right)^{\frac{1}{2}}  \left( \kP(y, y, \frac{t}{2}) \right)^{\frac{1}{2}}\int_{A} G(x, z) |V(z)| \dz\\
 & \leq C  \kP(x, y, t)  \int_{A} \left (G(x,z)+G(z, y)\right) |V(z)| \dz.
 \end{align*}
In light of \eqref{G_est}, we obtain
\begin{equation}\label{impstep2}
I _{2}^{\delta } \leq  C  \kP(x, y, t) \qquad \forall x, y \in A \mbox{ and } t \geq \gd_{0}.
\end{equation}
Hence, by adding estimates \eqref{impstep1} and \eqref{impstep2} we obtain
\begin{equation*}
S(V, x, y, t) \leq C \qquad \forall x, y \in A \mbox{ and }  t \geq \gd_{0},
\end{equation*}
where the constant $C$ depends on $\delta, A, p, N,$ and $\|V\|_p $\,.

\medskip

{\bf{Step 2}:} In this step we use our assumption that the diameter of $K:= \mbox{supp} \, V $ is `small enough', and estimate $S(V, x, y, t)$ for  $ t < \gd_{0}$ and $x, y \in A$, where $A$ is a `small' bounded domain with a smooth boundary containing $K$.
We use the short time behaviour of the heat kernel (see Lemma~\ref{shorttime}).

Denote by $g (x,y,t)$ the Gauss-Weierstrass type kernel
\begin{equation}\label{GWtype}
g (x,y,t) := \left( \frac{1}{4 \pi t} \right)^{\frac{N}{2}} \mathrm{e}^{-\frac{d(x,y)^2}{4t}}.
\end{equation}
Due to our assumptions on the smallness of $K$ and the smoothness of $P$ and $M$,  Lemma~\ref{shorttime} implies that there exist $\gd_0>0$ and $C>0$ such that
\begin{equation}\label{small}
C^{-1}g(x,y,t)\leq  \kP(x,y,t)\leq\!  C g(x,y,t) \qquad \forall x,y \in A\mbox{ and } t < \gd_{0}.
\end{equation}
Note that
\begin{equation}\label{property}
g(x,y, t)^{p} = g\Big(x,y, \frac{t}{p}\Big) (4 \pi t)^{\frac{(1 - p)N}{2}}(p)^{-\frac{N}{2}}\qquad \forall x,y,z\in M \mbox{ and } t>0.
\end{equation}
Following \cite{BDS}, and using  \eqref{small}, and \eqref{property}  for $x,y \in A$ and $0<t < \gd_{0}$, we obtain
\begin{multline*}
S(V, x, y, t)  :=  \int_{0}^{t} \int_{A}    \frac{ \kP(x, z, t-s) \kP(z, y, s)}{\kP(x, y, t)}  \  |V(z)|  \dz \ds  \\
 \leq    \int_{0}^{t} \int_{A}\frac{\left[\left(g(x, z, t-s)\right)^{p'} \left(g(z, y, s)\right)^{p'}\right]^{1/p'} |V(z)|}{g(x, y, t)} \dz \ds \\
\leq C  \int_{0}^{t} \left[ \frac{s(t-s)}{t} \right]^{-\frac{N}{2p}} \int_{A} \frac{\left[g\Big(x, z, \frac{t -s}{p'}\Big) g\Big(z, y, \frac{s}{p'}\Big)\right]^{1/p'}|V(z)|\dz}
{g\Big(x, y, \frac{t}{p'}\Big)^{1/p'}}   \ds .
\end{multline*}
Consequently, the H\"older inequality, \eqref{small}, the Chapman-Kolmogorov equation, and our assumption that $p>N/2$ imply that for all $x, y \in A$  and  $t < \gd_{0}$
\begin{multline*}
S(V, x, y, t)\!\leq \!C  ||V||_{p}\!\int_{0}^{t}\! \left[ \!\frac{s(t-s)}{t}\! \right]^{-\frac{N}{2p}}\!\! \dfrac{\left[\int_{A} g\Big(x, z, \frac{t -s}{p'}\Big) g\Big(z, y, \frac{s}{p'}\Big) \dz \!\right]^{1/p'}\!}{g\Big(x, y, \frac{t}{p'}\Big)^{1/p'}}   \ds\\
 \leq C  ||V||_{p}\int_{0}^{t} \left[ \frac{s(t-s)}{t} \right]^{-\frac{N}{2p}} \dfrac{\left[\int_{A} \kP\Big(x, z, \frac{t -s}{p'}\Big) \kP\Big(z, y, \frac{s}{p'}\Big) \dz \right]^{1/p'}}
{\kP\Big(x, y, \frac{t}{p'}\Big)^{1/p'}}   \ds \\[2mm]
 \leq C ||V||_{p}\int_{0}^{t}  \left[ \frac{s(t-s)}{t} \right]^{-\frac{N}{2p}} \ds = C ||V||_{p} t^{1- \frac{N}{2p}} \int_{0}^{1} \sigma^{-\frac{N}{2p}}(1- \sigma)^{-\frac{N}{2p}}  \dsigma \\
 \leq C||V||_{p} B \left(1 - \frac{N}{2p}, 1 - \frac{N}{2p} \right) t^{1- \frac{N}{2p}} \leq C,
\end{multline*}
where $B$ denotes the beta function.

\medskip

{\bf{Step 3}:} Steps 1 and 2, imply that the $3k$-inequality holds for
for all $x,y\in A$  and all $t > 0 $. Hence,  Lemma~\ref{maximum} implies that the $3k$-inequality holds for all $x,y\in M$  and all $t > 0 $. Consequently, Theorem~\ref{thm_3k_implies_eq} implies  that there exists $\vge_0>0$ such that $k_{P_\vge}^M\asymp \kP$ for all $|\vge|<\vge_0$.
\end{proof}
\begin{proof}[Proof of Theorem \ref{main}] (1) Let $V$ be the given potential with a compact support, and let $\{A_i\}_{i=1}^m$ be a finite open covering of $\mathrm{supp}\,V$ by smooth bounded sufficiently `small' domains $A_i$ such that $\kP$ satisfies
\begin{equation*}\label{smalli}
C^{-1}g(x,y,t)\leq  \kP(x,y,t)\leq\!  C g(x,y,t) \quad \forall x,y \in A_i,\;  t < \gd_{0} \mbox{ and } 1\leq i\leq m,
\end{equation*}
where $g$ is the Gauss-Weierstrass type kernel \eqref{GWtype}.

   Let  $\{\chi_i\}_{i=1}^m$ be a smooth partition of unity subordinated to this covering, and let
$V_i(x):=\chi_i(x)V(x)$. Then $V(x):=\sum_{i=1}^m \chi_i(x) V(x)=\sum_{i=1}^m V_i(x)$.

Using Proposition~\ref{mainprop}  $m$-times with $\vge$ small enough, we obtain that
\begin{equation*}
\kP  \asymp   k_{P -\varepsilon V_1}^{M} \asymp \ldots \asymp k_{P - \varepsilon (\sum_{i = 1}^{m -1} V_{i})}^{M}     \asymp k_{P_\varepsilon}^{M}.
\end{equation*}

\medskip

(2) The proof follows immediately  from assertion \mbox{(1)} and Lemma~\ref{lem1}.

\medskip

(3) Since $P-V$ is a subcritical operator with a heat kernel satisfying the doubling condition \eqref{assumption1}, we may apply part (1) of the theorem to the operator $P -V$ to conclude that there exists some $\tilde \varepsilon_{0}$ such that $k^{M}_{P - V} := k^{M}_{P_1} \asymp k^{M}_{P_{(1-\vge)}}$ for
all $|\vge| < \tilde \vge_{0}$ holds true. Therefore, there exist $\alpha$ and $\beta$ such that the hypotheses of Lemma~\ref{global-eq} are satisfied, and hence
$\kP \asymp k^{M}_{P -V}.$

\medskip

(4) Proposition~\ref{thm3kbounded}, Step~1 of the proof of Proposition~\ref{mainprop}, and Lemma~\ref{maximum} imply the $3k$-inequality.
Hence, by Theorem~\ref{thm_3k_implies_eq} there exists $\vge_0>0$ such that $k_{P_\vge}^M\asymp \kP$ for all $|\vge|<\vge_0$. Consequently, assertions (2) and (3) for a bounded compactly supported potential $V$ follow exactly as above.
\end{proof}
Conversely, it turns out that if $V\geq 0$, and $\kP\asymp k_{P+V}^M$, then the $3k$-inequality holds true. Indeed
\begin{prop}\label{cor_3k}
Let $V \geq 0$ and $\kP\asymp k_{P+V}^M$, then the heat kernel $\kP$ satisfies the $3k$-inequality \eqref{3k}.
\end{prop}
\begin{proof}
 Since $\kP\asymp k_{P+V}^M$, part~\mbox{(3)} of Theorem~\ref{thm_3k_implies_eq} implies that $k_{P+ V}^M$ satisfies the resolvent equation
\begin{multline*}
k_{P+  V}^{M}(x, y ,t)  \\= \kP(x, y, t) -  \int_{0}^{t} \int_{M} \kP(x, z, t-s) V(z) k_{P+  V}^{M} (z, y, s) \dz \ds.
\end{multline*}
Hence,
\begin{multline*}
\int_{0}^{t} \int_{M} \kP(x, z, t-s) V(z) k_{P}^{M} (z, y, s) \dz \ds\\ \leq   C\int_{0}^{t} \int_{M} \kP(x, z, t-s) V(z) k_{P+ V}^{M} (z, y, s) \dz \ds \leq C\kP(x, y, t)
\end{multline*}
for all $x,y\in M$ and $t>0$.
\end{proof}
\section{Small perturbation and the proof of Theorem~\ref{spmain}}\label{smp}
In the present section we introduce the class of {\em small perturbations} (see Definition~\ref{def_sp}), and prove Theorem~\ref{spmain} that extends Theorem~\ref{main} from the class of compactly supported perturbations to the class of small perturbations. In the context of Green functions the notion of small perturbation was introduced in \cite{YP1} and then was extended to the notion of semismall perturbation by M.~Murata in \cite{MM1} (see \cite{M-H_M,MM1,YP99} and references therein for some applications). Similarly to the elliptic case,  we study here the properties of small perturbations with respect to the heat kernel $\kP$.


Let $\{ M_{n}\}_{n = 0}^{\infty}$ be an exhaustion of $M$ as in the proof of Theorem~\ref{thm_3k_implies_eq}, and denote $M_{n}^{*} := M \setminus \overline{M_{n}}$.
Let $V$ be a given potential, and  $\{\Phi_{n}\}_{n = 0}^{\infty}$ be a sequence of smooth cutoff functions subordinated to the exhaustion $\{ M_{n}\}$ satisfying
$$
\Phi_{n}(x) =\left\{
               \begin{array}{ll}
                 1 & \hbox{if } x \in M_{n}, \\[2mm]
                 0 & \hbox{if } x \in M_{n+1}^* ,
               \end{array}
             \right.
$$
and $0 \leq \Phi_{n} \leq 1.$ Set $V_{n}(x) := \Phi_{n}(x) V(x)$ and $W_{n}(x) := V(x) - V_{n}(x).$
\begin{defi}\label{def_sp}{\em
We say that $V$ is a {\em small (resp., semismall) perturbation} with respect to the heat kernel $\kP$ if
\begin{equation}\label{smallperturbation}
\lim_{n \rightarrow \infty} \left\{ \sup_{\substack{ x, y \in M_{n}^{*}\\t>0}} \int_{0}^{t} \int_{M_{n}^{*}} \frac{\kP(x, z, t-s) |V(z)| \kP(z, y, s)}{\kP(x, y, t)} \dz \ds \right\} = 0,
\end{equation}
(resp.,
\begin{equation}\label{ssmallperturbation}
\lim_{n \rightarrow \infty} \left\{ \sup_{\substack{ y \in M_{n}^{*}\\t>0}} \int_{0}^{t} \int_{M_{n}^{*}} \frac{\kP(x_0, z, t-s) |V(z)| \kP(z, y, s)}{\kP(x_0, y, t)} \dz \ds \right\} = 0,
\end{equation}
where $x_0$ is a fixed reference point in $M$).
 }
\end{defi}
Clearly, if $V$ is a small perturbation with respect to $\kP$, then  it is also a semismall perturbation with respect to $\kP$ (see Subsection~\ref{subsec_op} for further discussions).


\begin{example}\label{spex1}{\em
Suppose that $P$ is a subcritical operator in $M$. Then a real valued function $V \in L^{p}(M)$, $p > \frac{N}{2}$ with {\em compact support} is a small
perturbation of $P$ with respect to $\kP$.
}
\end{example}
\begin{example}\label{spex2}{\em
Let $P:= -\Delta$ in $\mathbb{R}^{N}, N \geq 3$ and suppose that $V \in L^{p}(\mathbb{R}^{N}) \cap L^{q}(\mathbb{R}^{N})$, where $q < \frac{N}{2} < p$. It follows from \cite{BDS} that $V$ is a small perturbation
with respect to the Gauss-Weierstrass heat kernel \eqref{GW}. For example, for $t>2$ and $x,y\in \R^N$ we have
\begin{align*}
& \int_{0}^{t} \int_{M} \frac{\kP(x, z, t-s) |W_{n}(z)| \kP(z, y, s)}{\kP(x, y, t)} \dz \ds \\
& \leq c_{1}\|W_{n}\|_p \int_{0}^{1} s^{-N/(2p)} \ds + c_{2} \|W_{n}\|_q \int_{1}^{\infty} s^{-N/(2q)} \ds,
\end{align*}
and the dominated convergence theorem  implies that $V$ satisfies \eqref{smallperturbation}.
}
\end{example}
It turns out that under some further assumptions if $V$ is a small perturbation with respect to $\kP$, then $\kP$ satisfies the $3k$-inequality with respect to $V$. We have.
\begin{lem}\label{smallper1}
Suppose that the Riemannian manifold $(M, \mathfrak{g})$, the operator $P$, and its kernel $\kP$ satisfy the assumptions of Theorem~\ref{main}. Let $V\in L^p_\loc(M)$ be a small perturbation with respect to the heat kernel $\kP$, where $p>N/2$.
Then  $\kP$ satisfies the $3k$-inequality \eqref{3k} with respect to $V$.
 \end{lem}
\begin{proof}
Theorem~\ref{main} and Lemma~\ref{lem1} imply that $\kP \asymp k^{M}_{P + |V_{n+1}|}$ for any $n\in \mathbb{N}$.   Therefore, by Proposition~\ref{cor_3k}, for each $n\in \mathbb{N}$ there exists $C_n>0$ such that
\begin{multline}\label{3kVn}
\int_{0}^{t} \int_{M_{n}} \kP(x, z, t-s)  |V(z)| k_{P}^{M} (z, y, s) \dz \ds \\\leq
\int_{0}^{t} \int_{M} \kP(x, z, t-s)  |V_{n+1}(z)| k_{P}^{M} (z, y, s) \dz \ds
\leq C_n  \kP(x, y, t)
\end{multline}
for all $x,y\in M$ and $t>0$.


On the other hand, by the definition of a small perturbation we have
\begin{equation}\label{spc1}
\! \int_{0}^{t}\!\! \int_{M_n^*}\!\!\! k_{P}^{M}(x,z, t-s) |V(z)| k_{P}^{M}(z,y,s) \! \dz \ds\! \leq \! C k_{P}^{M}(x,y,t)
\end{equation}
for any $x, y \in M_{n}^*$ and $t>0$. So, by adding \eqref{3kVn} and \eqref{spc1}, we see that the $3k$-inequality \eqref{3k} holds true for $x,y\in M_n^*$ and $t>0$.

\medskip

Fix $y \in M_{n}^*$ and for $x\in M_n$ define
\[
U_{y}(x, t) := \int_{0}^{t}\!\! \int_{M_{n + 1}^*} \!\!\! k_{P}^{M}(x,z, t-s) |W_{n+1}(z)| k_{P}^{M}(z,y,s) \! \dz \ds.
\]
By \eqref{spc1} and continuity we have
\[
U_{y}(x, t) \leq \vge \kP(x, y, t) \quad \forall x \in \partial M_{n}, \ t > 0.
\]
On the other hand,
\[
U_{y}(x, 0) = \kP(x, y, 0)=0 \quad \forall x \in  M_{n}.
\]
 Moreover, $U_{y}$ satisfies the equation
\[
\frac{\partial}{\partial t} U_{y} + PU_{y} = |W_{n+1}(x)|\kP(x, y, t), \quad x\in M, \ t> 0.
\]
In particular,  $$\frac{\partial}{\partial t} U_{y} + PU_{y} = 0$$ for all $x \in M_{n}$.
Furthermore, the heat kernel $\kP(x, y, t)$, as a function $x$, also satisfies the equation
$$\frac{\partial}{\partial t} \kP+ P \kP= 0$$ for all $x \in M_{n}$. The  generalized maximum principle in $M_n$ implies that for any $y\in M_n^*$
\begin{equation}\label{eq_81}
U_{y}(x, t) \leq \vge \kP(x, y, t) \quad \forall x \in M_{n}, \ t > 0.
\end{equation}
Hence, taking into account \eqref{3kVn} it follows that the $3k$-inequality \eqref{3k} holds true for $x\in M_n$ and $y\in M_n^*$. The same argument shows that the $3k$-inequality holds true for
$y\in M_n$ and $x\in M_n^*$.

Suppose that $x, y \in M_{n}$. Then a similar comparison argument in $M_n$ shows that
\begin{equation}\label{sp2}
\int_{0}^{t}\!\! \int_{M_{n + 1}^*} \!\!\! k_{P}^{M}(x,z, t-s) |W_{n+1}(z)| k_{P}^{M}(z,y,s) \! \dz \ds \leq C \kP(x, y, t)
\end{equation}
for all $x, y \in M_{n}$ and $t>0$. Once again, taking into account \eqref{3kVn} it follows that the $3k$-inequality \eqref{3k} holds true also for $x\in M_n$ and $y\in M_n$.
Thus, the lemma is proved.
\end{proof}
In light of Part~\mbox{(3)} of Theorem~\ref{thm_3k_implies_eq} we obtain
\begin{cor}\label{cor_r_eq}
Suppose that $V$ is a (semi)small perturbation with respect to $P$ in $M$.
If $k_{P_\vge}^M \asymp k_{P}^M$ for some $\vge\in \R$, then  the heat kernel $k_{P_\vge}^M$ satisfies the resolvent equations \eqref{r_eq3}.
\end{cor}
Next, we prove Theorem~\ref{spmain}.
\begin{proof}[Proof of Theorem~\ref{spmain}]
Part (1) follows from Lemma~\ref{smallper1} and Theorem~\ref{thm_3k_implies_eq}.

\medskip

(2) The proof follows immediately from part (1) and Lemma~\ref{lem1}.

\medskip

(3) Since $W_n$ is a small perturbation of $\kP$, it follows from part (1) that for $n$ large enough the $3k$-inequality holds true with respect to $W_n$ with a constant $C<1$, and therefore,
\begin{equation}\label{kpeqkwn}
    \kP \asymp k^{M}_{P - W_{n}}.
\end{equation}
On the other hand, since  $P-V = (P- W_{n}) - V_{n}$ and $V_{n}$ has compact support in $M$,  it follows from \eqref{kpeqkwn} and part (3) of Theorem~\ref{main} that
\begin{equation*}\label{pf}
k^{M}_{P-V} \asymp k^{M}_{P-W_{n}} \asymp \kP. \qedhere
\end{equation*}
\end{proof}
\begin{rem}{\em
We note that if the heat kernels $\kP$ and $k_{P-V}^M$ are semi-equivalent for each fixed $x\in M$, then by  Davies-Harnack inequality and either by the short-time asymptotics of the heat kernels, or under the additional assumption that $V\in L^\infty(M)$, we get the equivalence of this heat kernels in $K\times M\times (0,\infty)$ for any  $K\Subset M$.
}\end{rem}
\begin{cor}\label{cor_sp}
Suppose that the operator $P$ and the potential $V$ satisfy the assumptions of Theorem~\ref{spmain}, and that $k_{P_\vge}^M$ satisfies \eqref{assumption1} for all $\vge\in S_+$. Then
$$S_+=\{\vge\in \R \mid k_{P_\vge}^M\asymp \kP\}.$$
\end{cor}
\section{Stability of the parabolic Martin boundary}\label{stableMartin}
In this section we study the behaviour of $\mathcal{C}_{L}(D)$, the cone of all nonnegative solutions of the parabolic equation
\begin{equation}\label{parabolic-cone}
Lu := \partial_{t} u + Pu = 0 \quad \mbox{in } D := M \times (a, b)
\end{equation}
under small perturbations, where $P$ is either of the form \eqref{P} or \eqref{operator} and  $-\infty \leq a < b \leq \infty$.

Our discussion is along the lines of the study of the elliptic case in \cite{MM1,YP3}, but the parabolic case needs a special care since the cone $\mathcal{C}_{L}(D)$ does not have a compact base. Before formulating the main result of the present section, we introduce some useful definitions and notations.
\begin{definition}{\em
Let $\mathcal{C}_{1} $ and  $\mathcal{C}_{2} $ be two convex cones embedded in topological spaces $\mathcal{V}_{1}$ and $\mathcal{V}_{2}$, respectively. $\mathcal{C}_{1}$ and $\mathcal{C}_{2}$ are said to be
\emph{affine equivalent} (and we denote it by $\mathcal{C}_{1} \cong \mathcal{C}_{2}$) if there exists homeomorphism $\Gf : \mathcal{C}_{1} \rightarrow \mathcal{C}_{2}$ which preserves convex combinations. Such a $\Gf$ is called an \emph{affine homeomorphism}.}
\end{definition}
Particularly, we are interested in the question whether $k_{P}^{M} \asymp k_{P-V}^{M} $  implies
$\mathcal{C}_{L}(D) \cong \mathcal{C}_{L-V}(D)$, where $V$ is a small perturbation.

We recall some of the basic facts concerning the parabolic Martin boundary and the parabolic Martin representation theorem (for more details see \cite{JLD,MM2}).

Let $x_{0}$ be a fixed reference point in $M.$ Consider a  nonnegative continuous function $\psi $ on $M$ such that $\psi(x) = 1$ on $B(x_{0}, r)$ and $\psi(x) = 0$ outside $B(x_{0}, 2r),$ for some $r > 0$ small enough. Also choose a nonnegative continuous $h$ on $(a, b)$ such that $h(t) = 0$ on $(a, a_1]$ and $h(t) > 0 $ on $(a_{1}, b)$, where $a < a_{1}  < b.$

Define a measure $\rho$ on $D$ by $\mbox{d} \rho(x, t):= \psi(x) h(t) \!\dx \!\dt .$ For any nonnegative measurable function $u$ on $D,$ we define
\begin{equation}\label{reference-measure}
\rho(u) := \int_{a}^{b}\int_{M}  u(x, t) \mbox{d} \rho (x, t).
\end{equation}
In the literature $\mathrm{d}\rho$ is known as a {\em reference measure}.

Define $$\mathcal{C}_{\rho, L}(D)\!:=\! \{ u \!\in\! \mathcal{C}_{L}(D) \mid \rho(u) \!<\! \infty \},\;\; \mathcal{C}_{\rho, L}^1(D) \!:=\! \{ u \!\in\! \mathcal{C}_{L}(D)\mid \rho(u) \!\leq\! 1\}.$$ Clearly, for every $u \in \mathcal{C}_{L}(D), $ there exists $h$ as defined above such that $\rho(u) < \infty.$
Hence, $\mathcal{C}_{L}(D)  = \cup_{\rho} \mathcal{C}_{\rho, L}(D)$. Moreover, the parabolic Harnack inequality implies that if $u\in \mathcal{C}_{\rho, L}(D)$ and $\rho(u) = 0$, then $u = 0$. Recall that nonnegative solution $u \in \mathcal{C}_{L}(D)$ is said to be {\em minimal} if for any nonnegative solution $v \in \mathcal{C}_{L}(D)$ such that $v \leq u$, there exists a nonnegative constant $c$ such that  $v = c u$. Denote by $\mathcal{C}_{L}^m(D)$ the set of all minimal solutions in $\mathcal{C}_{L}(D)$. By the Harnack principle, $\mathcal{C}_{\rho, L}^1(D)$ is a compact convex set in the compact-open topology, and by the Choquet theorem, the set of all extreme points of $\mathcal{C}_{\rho, L}^1(D)$ is equal to the union of the zero function and $\mathcal{C}_{L}^m(D)\cap \{u\in \mathcal{C}_{\rho, L}^1(D)\mid \gr(u)=1\}$.

We now introduce the Martin kernels. Let
$$
\kP ((x, t), (y, s)) :=
\left\{
  \begin{array}{ll}
    \kP(x, y, t-s) & a < s < t < b,\;\mbox{and } x,y\in M,\\[2mm]
    0 & a < t \leq s < b ,\; \mbox{and }  x,y\in M.
  \end{array}
\right.
$$
Fix a reference measure $\rho$, and define
$\mathcal{K}_{P}^{\rho} ((x, t), (y, s))$ the {\em (parabolic) $\rho$-Martin kernel} on ${D} \times {D}$ by
$$
\mathcal{K}_{P}^{\rho} ((x, t), (y, s)) :=
\left\{
  \begin{array}{ll}
    \dfrac{\kP((x, t), (y,s))}{\rho(\kP(\cdot \,, (y, s))} & a < s < t < b,\;\mbox{and } x,y\in M,\\[4mm]
    0 & a < t \leq s < b ,\; \mbox{and }  x,y\in M.
  \end{array}
\right.
$$
It follows that up to a homeomorphism, there exists a unique metrizable compactification $D_{L}^{\rho}$ of ${D}$ with the following properties (see for detail, \cite[Section~2]{MM2}):
\begin{itemize}
\item[(1)] The function $\mathcal{K}_{P}^{\rho} $ has a continuous extension to ${D} \times D_{L}^{\rho}$ such that for each
$(x, t) \in {D},$ the function $\mathcal{K}_{P}^{\rho} ((x, t),\, \cdot) $ is finite
valued and continuous on $D_{L}^{\rho} \setminus \{ (x, t) \} .$

\item[(2)] For $\gs_1, \gs_2\in D_{L}^{\rho}$ we have  $\mathcal{K}_{P}^{\rho}(\cdot\,, \gs_1) = \mathcal{K}_{P}^{\rho}(\cdot\,, \gs_2) $ if and only if $\gs_1= \gs_2$.
\end{itemize}

We write $\partial_{L}^{\, \rho}D := D_{L}^{\rho} \setminus {D} $, and we call it the  {\em parabolic $\rho$-Martin boundary of ${D}$ with respect to $L$ and a reference measure $\rho$}. A sequence  $\{ Y_{n} \} := \{(y_{n}, \tau_{n}) \}\subset D\times D $ is said to be {\em a fundamental sequence}
 if $\{ Y_{n} \}$ has no accumulation point in ${D} \times {D},$
$Y_{n} \rightarrow \sigma \in \partial_{L}^{\, \rho}D $. In particular, $\mathcal{K}_{P}^{\rho} ((x, t), (y_{n}, \tau_{n})) \rightarrow \mathcal{K}_{P}^{\rho} ((x, t), \sigma)$ locally uniformly in $D$, and
$\mathcal{K}_{P}^{\rho} (\cdot \,, \sigma)$ is a nonnegative solution to \eqref{parabolic-cone}. Note that by Fatou's Lemma, we have   $\rho(\mathcal{K}_{P}^{\rho}(\cdot \,, \sigma)) \leq 1$. So, $\mathcal{K}_{P}^{\rho} (\cdot \,, \sigma) \in \mathcal{C}_{\rho, L}^1(D)$ for any $\gs\in \partial_{L}^{\, \rho}D$.


We recall the parabolic Martin representation theorem. Define
\begin{equation}\label{minimal-set}
\partial_{L, 1}^{\,  \rho, m}D := \{ \sigma \in \partial_{L}^{\, \rho}D \mid \mathcal{K}_{P}^{\rho}(\cdot \,, \sigma)\in \mathcal{C}_{L}^m(D), \gr(\mathcal{K}_{P}^{\rho}(\cdot \,, \sigma))=1 \}.
\end{equation}
We call $\partial_{L, 1}^{\,  \rho, m}D $ the {\em (nontrivial) minimal parabolic  $\rho$-Martin boundary}.

The parabolic Martin representation theorem states: $u  \in \mathcal{C}_{\rho, L}(D)$,  if and only if there exists a unique Borel measure
 $\lambda$ on $\partial_{L}^{\,  \rho}D $ supported on $\mathcal{\partial}_{L, 1}^{\, \rho, m}D$, such that
 \begin{equation}\label{parabolic-represt}
 u(x, t) =    \int_{\mathcal{\partial}_{L}^{\, \rho}D} \mathcal{K}_{P}^{\rho}((x, t), \sigma)  \mbox{d} \lambda(\sigma),
 \end{equation}
and $\rho(u) = \lambda(\mathcal{\partial}_{L, 1}^{\, \rho}D).$


Next, we formulate our main result of the present section.
\begin{theorem}\label{martin-positive}
Let $P$ and $P - V$ be two subcritical operators such that $V$ is a small perturbation with respect to the heat kernel $\kP$, and $\kP\asymp k_{P- V}^{M}$ in $M\times M\times (0,\infty)$ with an equivalence constant $C$.

Then there exists an affine homeomorphism $\mathcal{T}:\mathcal{C}_{L}(D) \to \mathcal{C}_{L - V}(D)$ such that
\begin{equation}\label{extension3}
(\mathcal{T}u)(x, t) \!:=\!   u(x, t) +\!  \int_0^{t }\!\!\! \int_{M} \!\!\!\!k_{P\!-\! V}^{M}(x, z, t\! - \!s) V(z) u(z, s)   \!\dz\!\ds\ \;\; \forall u\!\in\! \mathcal{C}_{L}(D).
\end{equation}
Moreover,  for each $u\in   \mathcal{C}_{L}(D)$, we have $\mathcal{T}u \asymp u$ with equivalence constant $C^2$.
\end{theorem}
\begin{rem}{\em
In Theorem~\ref{martin-positive} we do not assume that $P$ is symmetric.
}
\end{rem}
\begin{rem}{\em
For the sake of brevity we present only the proof in the case $a = 0$ and  $ b = \infty$. So, we prove the case $D = M \times (0, \infty)$.
 It will be evident from the proof that all other cases follow along similar lines (see Remark~\ref{finalrem}).
}
\end{rem}
\begin{rem}{\em
Let $D = M \times (0, \infty)$. Then any fundamental sequence $\{ (y_{n}, \tau_{n}) \}$ converging to $\gs\in \partial_{L}^{\, \rho}D$ satisfies (up to a subsequence) $\tau_{n}\to T$, where  $0 \leq T \leq \infty$.

Therefore, if $T = \infty,$ i.e., $\tau_{n} \rightarrow \infty$, then  for a fixed $x\in M$ and $t>0$
 $$\kP(x, y_{n}, t - \tau_{n}) = k_{P - V}^{M}(x, y_{n}, t - \tau_{n}) = 0$$ for $n$ large enough, and therefore,
\[
\lim_{n\to\infty}\mathcal{K}_{P}^{\rho} ((x, t), (y_{n}, \tau_{n})) =  \lim_{n\to\infty}\mathcal{K}_{P- V}^{\rho} ((x, t), (y_{n}, \tau_{n})) = 0.
\]
On the other hand, if $\tau_{n} \rightarrow T$, where $ 0 < T < \infty$, then
the Martin kernel $\mathcal{K}_{P}^{\rho}((x, t), \sigma)$ satisfies $\mathcal{K}_{P}^{\rho}((x, t), \sigma)=0$ for all $t \leq T$. Hence, if the uniqueness of the positive Cauchy problem holds true, then $\mathcal{K}_{P}^{\rho}(\cdot, \sigma)=0$ in $D$.
}
\end{rem}
 The proof of Theorem~\ref{martin-positive} hinges on the following key proposition.
 \begin{prop}\label{minimal_martin_prop}
 Let $P$ and $\tilde P$ be two subcritical operator such that $k^{M}_{P} \asymp k^{M}_{\tilde P}$ in $M \times M \times (0, \infty).$ Then there exists a homeomorphism
 $\alpha_{\rho} : \partial_{L, 1}^{\rho, m} D \rightarrow \partial_{\tilde L, 1}^{\rho, m} D $ and $C > 0$ such that
 \begin{equation}\label{minimal_martin_eq}
 C^{-1} \mathcal{K}_{P}^{\rho} ((x, t), \sigma) \leq \mathcal{K}_{\tilde P}^{\rho} ((x, t), \alpha_{\rho}(\sigma)) \leq C \mathcal{K}_{P}^{\rho} ((x, t), \sigma)
 \end{equation}
 for every $\sigma \in \partial_{L, 1}^{\rho, m} D$ and $(x, t) \in M \times (0, \infty).$
 \end{prop}
For the proof of the above proposition we need the following lemma.
\begin{lem}\label{positive_martin_sol}
Suppose that $k^{M}_{P} \asymp k^{M}_{\tilde P}$. Then for every $u \in C^{1}_{\rho, L}(D)$ there exists $\tilde u \in C^{1}_{\rho, \tilde L}(D)$ that satisfies
$$
C^{-2} u(x, t) \leq \tilde u(x, t) \leq C^2 u(x, t) \quad (x, t) \in M \times (0, \infty)
$$
where $C$ is the equivalent constant for $k^{M}_{P}$ and $k^{M}_{\tilde P}$.
\end{lem}
The proof of Lemma~\ref{positive_martin_sol} is similar to the proof of \cite[Lemma~2.4]{YP3}, and therefore it is omitted.
\begin{proof}[Proof of Proposition~\ref{minimal_martin_prop}]
Let $\sigma \in \partial_{L, 1}^{\rho, m} D$, and for $k = 1, 2$, let $\{ (y_n^k, \tau_n^k) \}$ be two fundamental subsequences of a  fundamental sequence $\{ (y_n, \tau_n) \}\subset D$  such that
 $$
 (y_n^k, \tau_n^k) \rightarrow \sigma \quad \mbox{in} \ D_{L}^{\rho}, \quad \mbox{and } \quad
 (y_n^k, \tau_n^k) \rightarrow \tilde \sigma_{k}  \quad \mbox{in} \ D_{\tilde L}^{\rho} .
$$

We claim that $\tilde \sigma_{1} = \tilde \sigma_{2} $, and $ \tilde \sigma_{1} \in \partial_{L, 1}^{\rho, m} D$. In particular, the mapping $\sigma\mapsto \tilde \sigma_{1}$ is a well defined mapping $\alpha_{\rho} : \partial_{L, 1}^{\rho, m} D \rightarrow \partial_{\tilde L, 1}^{\rho, m} D $, defined  by
$\alpha_{\rho}(\sigma) := \tilde \sigma$, if $(y_n, \tau_n) \rightarrow \sigma \in \partial_{L, 1}^{\rho, m} D$, and $(y_n, \tau_n) \rightarrow \tilde \sigma \in \partial_{\tilde L, 1}^{\rho, m} D$.

Indeed, from our assumption that $k^{M}_{P} \asymp k^{M}_{\tilde P}$ it follows that
  \begin{equation}\label{minimal_martin_eq_1}
 C^{-2} \mathcal{K}_{P}^{\rho} ((x, t), \sigma) \leq \mathcal{K}_{\tilde P}^{\rho} ((x, t), \tilde \sigma_{k}) \leq C^2 \mathcal{K}_{P}^{\rho} ((x, t), \sigma)\quad \forall (x,t)\in D,
 \end{equation}
where $C$ is the equivalence constant.  Using \eqref{minimal_martin_eq_1}, we obtain $$
\mathcal{K}_{\tilde P}^{\rho}((x, t), \tilde \sigma_1 ) - C^{-4} \mathcal{K}_{\tilde P}^{\rho}((x, t), \tilde \sigma_2 ) \geq 0.
$$
We use now the maximal $\vge$ trick. Define
$$
\varepsilon_0 := \mbox{max} \{ \varepsilon > 0 :
\mathcal{K}_{\tilde P}^{\rho}((x, t), \tilde \sigma_1 ) - \varepsilon \mathcal{K}_{\tilde P}^{\rho}((x, t), \tilde \sigma_2 ) \geq 0 \},
$$
and let
$$
\tilde v_{\rho}(x, t) := \mathcal{K}_{\tilde P}^{\rho}((x, t), \tilde \sigma_1 ) - \varepsilon_0 \mathcal{K}_{\tilde P}^{\rho}((x, t), \tilde \sigma_2 ).
$$
Clearly $\tilde v_{\rho} \geq 0,$ and we may assume that $\rho(\tilde v_{\rho}) > 0$, since otherwise, $\tilde \sigma_1 = \tilde \sigma_2$. Lemma~\ref{positive_martin_sol} implies that there exists $u \in C^{1}_{\rho, L}(D)$ such that
\begin{equation}\label{minimal_eq_1}
C^{-2} u(x, t) \leq \frac{\tilde v_{\rho}(x, t)}{\rho(\tilde v_{\rho})} \leq C^2 u(x, t).
\end{equation}
Therefore, $0 \leq u(x, t) \leq C^4 (\rho(\tilde v_{\rho}))^{-1} \mathcal{K}^{\rho}_{P}((x, t), \sigma)$. Since $\mathcal{K}^{\rho}_{P}((x, t), \sigma)$ is a minimal solution, we have $u(x, t) = \mu \mathcal{K}^{\rho}_{P}((x, t), \sigma)$ for some $\mu > 0.$ By substituting this in \eqref{minimal_eq_1}, we obtain
$$
C^{-4} \mu \rho(\tilde v_{\rho}) \mathcal{K}^{\rho}_{\tilde P} ((x, t), \tilde \sigma_2) \leq C^{-2} \mu \rho(\tilde v_{\rho}) \mathcal{K}^{\rho}_{ P} ((x, t), \sigma) \leq \tilde v_{\rho}(x, t).
$$
Thus, by letting $\mu_0: = C^{-4} \mu \rho(\tilde v_{\rho}) > 0$, we obtain
$$
0 \leq \tilde v_{\rho}(x, t) - \mu_0 \mathcal{K}^{\rho}_{\tilde P} ((x, t), \tilde \sigma_2)
= \mathcal{K}_{\tilde P}^{\rho}((x, t), \tilde \sigma_1 ) - (\varepsilon_0 + \mu_0) \mathcal{K}_{\tilde P}^{\rho}((x, t), \tilde \sigma_2 ),
$$
which contradicts the definition of $\varepsilon_0$. Hence, $\tilde \sigma_1 = \tilde \sigma_2 $, and therefore, $\alpha_{\rho}$ is well defined. Moreover, \eqref{minimal_martin_eq_1} and Lemma~\ref{positive_martin_sol}, and the maximal $\vge$ trick imply that
$\tilde \sigma_1 \in \partial_{\tilde{L}, 1}^{\rho, m} D$, so $\alpha_{\rho} : \partial_{L, 1}^{\rho, m} D \rightarrow \partial_{\tilde L, 1}^{\rho, m} D $. By similar arguments,
$\alpha_{\rho}$ is injective, surjective and homeomorphism.
\end{proof}
We can now prove Theorem~\ref{martin-positive}.
\begin{proof}[Proof of Theorem~\ref{martin-positive}] 
Let $\{ M_{j}\}_{j = 0}^{\infty}$ be an exhaustion of $M$, and denote $M_{j}^{*} := M \setminus \overline{M_{j}}$.
Let $ Y_{n} = \{ (y_{n}, \tau_{n}) \}$ be a fundamental sequence converging to  $\sigma \in  \partial_{ L, 1}^{\rho, m} D$, and $\gt_n\to T$.

Fix $\vge>0$, and $x$ in $M$ and $t>0$.  Since
$V$ is a small perturbation with respect to $\kP$,  and since  $\kP$ is equivalent  to $k_{P - V}^{M}$, it follows from \eqref{eq_81} that there exists $j(\vge)$, and $n(\vge)$ such that  for $j>j(\vge)$, and $n>n(\vge)$, we have $ y_{n} \in M_{j(\vge)}^{*}$,  and for $t > \gt_n$ the following inequality holds
\begin{multline}\label{martin-eq-1}
\int_{\tau_{n}}^{t} \int_{M_{j}^{*}}   \frac{k_{P - V}^{M}(x, z, t  - s) |V(z)| \kP(z, y_{n}, s-\tau_{n})}{\kP(x, y_{n}, t - \tau_{n})}  \dz\ds\\[2mm]
=\int_{0}^{t - \tau_{n}} \!\!\!\!\int_{M_{j}^{*}}   \frac{k_{P - V}^{M}(x, z, t -\tau_{n} - \tilde{s}) |V(z)| \kP(z, y_{n}, \tilde{s})}{\kP(x, y_{n}, t - \tau_{n})}  \dz\,\mathrm{d}\tilde{s}  <   \vge.
\end{multline}
Since 
$$\lim_{n\to\infty}\frac{\kP(x, y_{n}, t - \tau_{n})}{\rho(\kP(\cdot \,, (y_{n},  \tau_{n})))}=\mathcal{K}_{P}^{\rho}((x, t), \sigma),$$
it follows that
$$
 \int_{\tau_{n}}^{t}\! \int_{M_{j}^{*}}   \frac{k_{P - V}^{M}(x, z, t  - s)
|V(z)| \kP(z, y_{n}, s-\tau_{n})}{\rho(\kP(\cdot \,, (y_{n},  \tau_{n})))} \dz\ds
   \leq   \vge M.
$$
Hence, the sequence of functions
$$\left\{f_n(z,s):= k_{P- V}^{M}(x, z, t - s) V(z) \frac{\kP(z, y_{n}, s-\tau_{n})}{\rho(\kP(\cdot \,, (y_{n}, \tau_{n})))}\right\}$$
is uniformly integrable and tight,  and
$$\lim_{n\to\infty}f_n(z,s)=k_{P- V}^{M}(x, z, t - s) V(z) \mathcal{K}_{P}^{\rho}((z, s), \sigma)\qquad \mbox{locally uniformly}.$$
In light of Corollary~\ref{cor_r_eq}, the resolvent equation implies
\begin{multline}\label{martin-duhamel}
\frac{k_{P - V}^{M}(x, y_{n}, t- \tau_{n})}{\rho(\kP(\cdot \,, (y_{n}, \tau_{n})))}   = \frac{\kP(x, y_{n}, t- \tau_{n})}{\rho(\kP(\cdot \,, (y_{n}, \tau_{n})))}   \\[2mm]
 + \!\int_{\tau_{n}}^{t}\!\! \int_{M}\!\! \frac{k_{P- V}^{M}(x, z, t - s) V(z) \kP(z, y_{n}, s-\tau_{n})}{\rho(\kP(\cdot \,, (y_{n}, \tau_{n})))}  \dz\!\ds.
\end{multline}
Hence, by the Vitali convergence theorem (\cite[p.~98]{RF}) we may pass to the limit to obtain
\begin{multline}\label{martin7}
\lim_{n \rightarrow \infty} \frac{k_{P - V}^{M}(x, y_{n}, t- \tau_{n})}{\rho(\kP(\cdot \,, (y_{n}, \tau_{n})))}   = \mathcal{K}_{P}^{\rho}((x, t), \sigma)  \\
 +  \int_T^{t }\int_{M}  k_{P- V}^{M}(x, z, t - s) V(z) \mathcal{K}_{P}^{\rho}((z, s), \sigma)  \dz\ds.
\end{multline}
Furthermore, since $\kP$ is equivalent to $\asymp k_{P-V}^M$, we may define (up to a subsequence)
\begin{align*}
\mathcal{K}_{P - V}^{\rho}((x, t), \alpha_{\rho}(\sigma))&:=
\lim_{n \rightarrow \infty} \frac{k_{P - V}^{M}(x, y_{n}, t- \tau_{n})}{\rho(k_{P-V}^M(\cdot \,, (y_{n}, \tau_{n}))}\, \in  \partial_{L-V, 1}^{\rho, m} D,   \\[2mm]
\mbox{ and }  \lambda_{\rho}(\sigma)&:=\lim_{n \rightarrow \infty} \frac{\rho(k_{P- V}^{M}(\cdot \,, (y_{n}, \tau_{n}))}{\rho( k_{P}^{M}(\cdot \,, (y_{n}, \tau_{n}))}\,,
\end{align*}
where $C^{-1} \leq \lambda_{\rho}(\sigma) \leq C$. Moreover, Proposition~\ref{minimal_martin_prop} implies that $\alpha_{\rho}(\sigma)$ is well defined, and consequently, the sequence $\{ (y_{n}, \tau_{n}) \}$ converges in
$D_{L - V}^{\rho}$ to the point $\alpha_{\rho}(\sigma) \in  \partial_{L-V, 1}^{\rho, m} D$.
Therefore, $\lambda_{\rho}(\sigma)$ does not depend on the subsequence.

Consequently, the following resolvent equation for minimal Martin functions holds true
\begin{multline}\label{M_w_d}
\lambda_{\rho}(\sigma) \mathcal{K}_{P  -  V}^{\rho}((x, t), \alpha_{\rho}(\sigma))   =  \mathcal{K}_{P}^{\rho}((x, t), \sigma)\\
  +  \int_{T}^{t }\int_{M}      k_{P -  V}^{M}(x, z, t  -  s) V(z) \mathcal{K}_{P}^{\rho}((z, s), \sigma)     \dz\ds.
\end{multline}
But since $\mathcal{K}_{P}^{\rho}((z, s), \sigma)=0$ for $0\leq s\leq T$, we have
\begin{multline}\label{M_w_d1}
\lambda_{\rho}(\sigma) \mathcal{K}_{P  -  V}^{\rho}((x, t), \alpha_{\rho}(\sigma))   =  \mathcal{K}_{P}^{\rho}((x, t), \sigma)\\
  +  \int_{0}^{t }\int_{M}      k_{P -  V}^{M}(x, z, t  -  s) V(z) \mathcal{K}_{P}^{\rho}((z, s), \sigma)     \dz\ds.
\end{multline}
 Define $$\mathcal{T}_{\rho} :  \{  \mathcal{K}_{P}^{\rho}(\cdot \,, \sigma)\mid \sigma \in  \partial_{L-V, 1}^{\rho, m} D \} \rightarrow  \mathcal{C}_{\rho,L-V}(D)$$ by
$$\mathcal{T}_{\rho}\mathcal{K}_{P}^{\rho}((x,t) \,, \sigma) := \lambda_{\rho}(\sigma) \mathcal{K}_{P - V}^{\rho}((x, t), \alpha_{\rho}(\sigma)).$$

Extend $\mathcal{T}_{\rho}$ to an affine transformation (with a slight abuse of notation)
$$\mathcal{T}_{\rho} : \mathrm{Conv}(\{  \mathcal{K}_{P}^{\rho}(\cdot \,, \sigma)\mid \sigma \in  \partial_{L-V, 1}^{\rho, m} D \})\rightarrow  \mathcal{C}_{\rho, L-V}(D) ,$$
 where $\mathrm{Conv}(A)$ is the convex hull of a set $A$. Then, using the parabolic Martin representation theorem and a standard continuity arguments (follows from continuity of the Martin kernel $\mathcal{K}^{\rho}_{P}(\cdot\,, \sigma))$,
  we extend $\mathcal{T}_\rho$ to a continuous affine transformation $ \mathcal{T}_{\rho} : \mathcal{C}_{\rho,L}(D) \rightarrow  \mathcal{C}_{\rho,L-V}(D) $ given by
\begin{equation}\label{extension}
(\mathcal{T}_{\rho}u)(x, t) :=   u(x, t) +  \int_0^{t }\int_{M}  k_{P- V}^{M}(x, z, t - s) V(z) u(z, s) \  \dz\ds.
\end{equation}
Recall that $\mathcal{C}_{L}(D) = \cup_{\rho} \mathcal{C}_{\rho, L}(D)$. Moreover, the mapping $\mathcal{T}_{\rho}$  given by \eqref{extension} does not depend on $\gr$. Therefore, we may extend the family of transformations $\{\mathcal{T}_{\rho}\}_\gr$  to a continuous affine transformation $ \mathcal{T} : \mathcal{C}_{L}(D) \rightarrow  \mathcal{C}_{L-V}(D) $  by
$\mathcal{T}u:=\mathcal{T}_{\rho}u$ for $u\in \mathcal{C}_{\rho, L}(D)$, so, we get \eqref{extension3}.


Analogously, define $  \mathcal{S} : \mathcal{C}_{L- V}(D) \rightarrow  \mathcal{C}_{L}(D) $ by
 \begin{equation}\label{extension1}
( \mathcal{S}v)(x, t) :=   v(x, t) -  \int_0^{t }\int_{M}  k_{P}^{M}(x, z, t - s) V(z) v(z, s)  \dz\ds.
\end{equation}
We claim that $\mathcal{S}\mathcal{T} = \mathrm{Id}_{\mathcal{C}_{L}(D)}$ and $\mathcal{T} \mathcal{S} = \mathrm{Id}_{\mathcal{C}_{L-V}(D)}$, where $\mathrm{Id}_A$ is the identity map on the set $A$. We show that $\mathcal{S} \mathcal{T} = \mathrm{Id}_{\mathcal{C}_{L}(D)}$ and the second assertion follows similarly.

For $u \in \mathcal{C}_{L}(D)$ we have
\begin{align*}
&(\mathcal{S} \mathcal{T} u)(x, t)   = \mathcal{S} \left( u(x, t) + \int_0^{t}\int_{M}  k_{P - V}^{M}(x, y, t-\alpha) V(y) u(y, \alpha)   \dy\da \right) \\
& = u(x, t) + \int_0^{t}\int_{M}  k_{P - V}^{M}(x, y, t-\alpha) V(y) u(y, \alpha)   \dy\da \\
& - \int_0^{t}\int_{M}  k_{P }^{M}(x, y, t-\alpha) V(y) u(y, \alpha)   \dy\da \\
& -\!\! \int_0^{t}\!\!\int_{M} \!\! k_{P}^{M}(x, y, t\!-\!\alpha) V(y)\!
\left(\int_{0}^{\alpha}\!\!\int_{M}  \!\!k_{P\! -\! V}^{M}(y, z, \alpha \!-\! s) V(z) u(z, s) \ \! \dz\!\ds  \!\right)\!\! \dy\da.
\end{align*}
Using Fubini's theorem and the resolvent equation for the heat kernel $k_{P-V}^M$, we obtain
\begin{align*}
 &\int_0^{t}\int_{M}  k_{P - V}^{M}(x, z, t- \ga) V(z) u(z, \ga)   \dz\da\\
 &=\int_0^{t}\int_{M}  k_{P }^{M}(x, y, t-\alpha) V(y) u(y, \alpha)   \dy\da \\
& +\!\! \int_0^{t}\!\!\int_{M} \!\! k_{P}^{M}(x, y, t\!-\!\alpha) V(y)\!
\left(\int_{0}^{\alpha}\!\!\int_{M}  \!\!k_{P\! -\! V}^{M}(y, z, \alpha \!-\! s) V(z) u(z, s) \  \! \dz\!\ds  \!\right)\!\! \dy\da.
\end{align*}
 Thus,
$(\mathcal{S} \mathcal{T} u)(x, t)  =  u(x, t)$.
\end{proof}
\begin{rem}\label{finalrem}{\em In the general case, where
 $D = M \times (a, b)$, with $-\infty \leq a<b\leq\infty$, the transformations  $\mathcal{T}$ and $\mathcal{S}$, given  by \eqref{extension} and \eqref{extension1} (with $a$ replacing $0$), are well defined affine homeomorphisms even if $a=-\infty$ (thanks to the $3k$-inequality (see Lemma~\ref{smallper1})).
 }
 \end{rem}
\section{Concluding remarks}\label{concludingrem}
This section consists of three subsections. In the first one, we briefly extend our results to a certain class of {\em nonsymmetric} operators, while in Subsection~\ref{subsec_ex} we provide several examples to illustrate our results. Finally, in Subsection~\ref{subsec_op} we pose some open problems.
\subsection{Quasi-symmetric heat kernels}\label{subsec_qso}
The positive minimal heat kernel $\kP$ is said to be \emph{quasi-symmetric} if
\begin{equation}\label{quasi-symm}
\kP(x, y, t) \asymp \kP(y, x, t) \qquad \forall x, y \in M, \ t > 0.
\end{equation}
\begin{rem}\label{rem_ancona}{\em
In \cite{Ancona} A.~Ancona introduced the notion of quasi-symmetric operators (with respect to its Na\"{i}m kernel). Clearly, if the heat kernel $\kP$ is quasi-symmetric, and the operator $P$ is subcritical, then $P$ is quasi-symmetric in the sense of Ancona.
}
\end{rem}
\begin{lem}\label{quasi-monotone}
Suppose that the heat kernel $\kP$ is quasi-symmetric. Then there exists a constant $C > 0$ such that
\begin{equation}\label{monotone-quasi}
\kP(x, y, t) \leq C \left( \kP(x, x, t) \right)^{\frac{1}{2}}\left( \kP(y, y, t) \right)^{\frac{1}{2}} \qquad \forall x, y \in M, \ t > 0.
\end{equation}
\end{lem}
\begin{proof}
Using the Chapman-Kolmogorov equation and the H\"older inequality, we see that
\begin{align*}
& \kP(x, y, t)   = \int_{M} \kP (x, z, \frac{t}{2}) \kP(z, y, \frac{t}{2}) \dz  \\
& \leq \left( \int_{M} \left( \kP (x, z, \frac{t}{2}) \right)^2  \dz \right)^{\frac{1}{2}} \left( \int_{M} \left( \kP (z, y, \frac{t}{2}) \right)^2  \dz \right)^{\frac{1}{2}} \\
& \leq C \left( \int_{M}  \kP (x, z, \frac{t}{2}) \kP(z, x, \frac{t}{2})  \dz \right)^{\frac{1}{2}} \left( \int_{M}  \kP (y, z, \frac{t}{2}) \kP(z, y, \frac{t}{2})  \dz \right)^{\frac{1}{2}} \\
& = C \left( \kP(x, x, t) \right)^{\frac{1}{2}} \left( \kP(y, y, t) \right)^{\frac{1}{2}}.
\qedhere
\end{align*}
\end{proof}
\begin{defi}{\em
The heat kernel $\kP$ is said to be {\em quasi-monotone} at $x_{0} \in M$ if for any  $T > 0$ there exists $C:= C(x_{0}, T) > 0$ such that
\[
 \kP(x_{0}, x_{0}, t_{2}) \leq C \kP(x_{0}, x_{0}, t_{1}), \quad  \forall \  t_{2} \geq t_{1} > T.
 \]
}
\end{defi}
Clearly, the heat kernel of a symmetric operator is quasi-symmetric and also quasi-monotone at all $x\in M$.
\begin{rem}{\em Suppose that $\kP$ is quasi-symmetric and also quasi-monotone at a point $x_0\in M$. Following the proof of Davies in \cite[Theorem~10]{DA1}, it follows that such $\kP$ satisfies  the Davies-Harnack inequality \eqref{eqharnack}.
In light of Lemma~\ref{quasi-monotone}, we can analogously deduce theorems~\ref{main}, \ref{spmain}, (and hence also Theorem~\ref{martin-positive}) for the class of  \emph{quasi-symmetric} heat kernels which are  \emph{quasi-monotone} (and satisfy \eqref{assumption1}).
}
\end{rem}
\begin{rem}
{\em
It should be noted that we are unaware of any example of a nonsymmetric operator whose heat kernel is quasi-symmetric but whose heat kernel is not equivalent to a symmetric
one. Conversely, if the heat kernel of any nonsymmetric operator $P$ is equivalent to the heat kernel of a symmetric operator in $M$, then the heat kernel of $P$
is quasi-monotone at any point $x_0\in M$, and quasi-symmetric (and $P$ is  quasi-symmetric as well).
}
\end{rem}
\subsection{Examples}\label{subsec_ex}
In the present subsection we give various examples of Riemannian manifolds $M$ and heat kernels $\kP$ defined on $M$ which satisfy our main assumption \eqref{assumption1} of  theorems~\ref{main} and \ref{spmain} (the doubling condition). Hence, our main results of the paper apply to these cases.

The study of heat kernel estimates has a long history (see for example \cite{DA, greg1,Ouhabaz}. In particular, proving pointwise two-sided Gaussian estimates for the heat kernel was a subject of intense research for the past few decades. It started with the celebrated works of Nash \cite{JN} and Aronson~\cite{AR}, where two-sided Gaussian estimates were obtained for the heat kernel of a uniformly elliptic operator in divergence form in $\R^N$. For such operators we obtain:
\begin{example}\label{ex1}{\em
Consider a parabolic equation of the form $\frac{\partial u}{\partial t} + P u = 0$ on $\R^N\times (0,\infty)$,  where $N\geq 3$ and
\begin{equation}\label{divform1}
P = -\sum_{i, j = 1}^{N} \frac{\partial}{\partial x_{i}} \left( a_{ij}(x) \frac{\partial}{\partial x_{j}} \right)
\end{equation}
is a uniformly elliptic operator with real, bounded coefficients satisfying the assumptions of Theorem~\ref{main}. Denote by $k_P^{\R^N}$ the corresponding positive minimal heat kernel.  Aronson
\cite[Theorem~7]{AR} proved that $k_P^{\R^N}$ admits
two sided Gaussian estimates, i.e., there exist positive constants $C_{1}, C_{2}, C_{3}, C_{4} $ such that
\begin{equation}\label{twosided}
\frac{C_{3}}{t^{N/2}}\exp\left(- \frac{|x - y|^2}{C_{4}t}\right) \leq k_P^{\R^N}(x, y, t) \leq  \frac{C_{1}}{t^{N/2}}\exp\left(- \frac{|x - y|^2}{C_{2}t}\right)
\end{equation}
for all $x\in \R^N$ and $t > 0$. Estimate  \eqref{twosided} readily implies that
\[
k_P^{\R^N}(x, x, \frac{t}{2}) \leq C 2^{\frac{N}{2}} k_P^{\R^N} (x, x, t) \qquad \forall x\in \R^N \mbox{ and } t > 0,
\]
and hence, $k_P^{\R^N}$ satisfies the doubling condition \eqref{assumption1}. Therefore, if $V$ is a small perturbation of $k_P^{\R^N}$, then there exists $\varepsilon_{0} > 0$ such that $k_{P_\vge}^{\R^N} \asymp k_P^{\R^N}$ for all $|\vge| <  \vge_{0}$.
}
\end{example}
\begin{example}[Periodic operator]\label{ex11}{\em
Consider a uniformly elliptic operator $P$ on $\R^N$, $N\geq 3$ of the form
\[
P = -\sum_{i, j = 1}^{N} \frac{\partial}{\partial x_{i}} \left( a_{ij}(x) \frac{\partial}{\partial x_{j}} \right)+U(X).
\]
Assume that $P\geq 0$ in $\R^N$, and that the coefficients of $P$ satisfy the assumptions of Theorem~\ref{main}. Suppose that the coefficients of $P$ are {\em periodic} in $x_1,\ldots ,x_n$ with period $1$. Without loss of generality we may assume that $\gl_0(P,\mathbf{1},\R^N)=0$. Then the equation $Pu=0$ in $\R^N$ admits a unique (up to a multiplicative constant) positive solution $\gf$. Moreover, (in the symmetric case) $\gf$ is periodic in $x_1,\ldots ,x_n$ with period $1$ \cite{Ag}.

Using the ground state transform we get the operator
$$P_\gf:=(\gf)^{-1}P\gf=-(\gf)^{-2}(x)\sum_{i, j = 1}^{N} \frac{\partial}{\partial x_{i}} \left( \gf^2(x) a_{ij}(x) \frac{\partial}{\partial x_{j}} \right),$$
whose heat kernel satisfies $k_{P_\gf}^{\R^N}(x,y,t)=(\gf)^{-1}(x)k_{P}^{\R^N}(x,y,t)\gf(y)$.

Consequently, $P_\gf$ is, in fact, of the form \eqref{divform1} on $L^2(\R^N,\vgf^2\dx)$, and therefore, $k_{P_\gf}^{\R^N}$ satisfies assumption \eqref{assumption1} which in turn implies that $k_P^{\R^N}$ also satisfies \eqref{assumption1}.
Therefore, if $V$ is a small perturbation of $k_P^{\R^N}$, then there exists $\varepsilon_{0} > 0$ such that $k_{P_\vge}^{\R^N} \asymp k_P^{\R^N}$ for all $|\vge| <  \vge_{0}$.
}
\end{example}
Next, we consider perturbations of the Laplace-Beltrami operators
on noncompact Riemannian manifolds. Following the seminal work of Aronson, the question of estimating the heat kernel on Riemannian manifolds was extensively studied by many authors. One of the most general estimates
of heat kernels $\kP$ for the Laplace-Beltrami operators was proved by P.~Li and S.~T.~Yau \cite[Corollary~3.1 and Theorem~4.1]{LY} under a suitable curvature assumption. We use these celebrated  results in the following example.
\begin{example}\label{ex2}{\em
Let $(M, g)$ be a complete, connected, noncompact Riemannian manifold of dimension $N$ with nonnegative Ricci curvature. Let $ P := -\Delta_{g}$ denote the (positive) Laplace-Beltrami operator on $M$ and let $\kP$ denote the corresponding heat kernel. Then by \cite[Corollary~3.1 and Theorem~4.1]{LY} there exist positive  constants $C_{1}, C_{2}, C_{3}, C_{4}$ such that
\begin{equation}\label{curvatureestimate}
\frac{C_{3}}{V(x, \sqrt t)}\,\mathrm{e}^{\left(-\frac{d^2(x, y)}{C_{4}t}\right)} \leq \kP(x, y, t) \leq \frac{C_{1}}{V(x, \sqrt t)}\,\mathrm{e}^{\left(-\frac{d^2(x, y)}{C_{2}t}\right)}
\end{equation}
for all $x, y \in M$ and $t > 0$, where $d(x, y)$ is the geodesic distance on $M$ and $V(x, r)$ is the Riemannian volume of the geodesic ball $B(x, r) = \{ y \in M : d(x, y) < r \}$.
Moreover, under the above assumptions, $M$ satisfies the {\em doubling volume property} \eqref{DC} (see \cite[Theorem~15.21]{greg0}), and hence, \eqref{assumption1} is satisfied.


Alternatively, under the above assumptions E.~B.~Davies proved \cite[Corollary~5.3.6]{DA} that the positive minimal heat kernel $\kP$ satisfies the following  global exponential-type upper bound
\begin{equation*}
\kP(x, x, t + s) \leq \kP(x, x, t) \leq \kP(x, y, t + s) \left( \frac{t + s}{t} \right)^{\frac{N}{2}} \mathrm{e}^{\frac{d(x, y)^2}{4s}}\qquad \forall t,s>0.
\end{equation*}
In particular, for $ t = s$, we have
\begin{equation}\label{EBD}
\kP(x, x, t) \leq 2^{\frac{N}{2}} \kP(x, x, 2t)\qquad \forall t>0.
\end{equation}
Hence, \eqref{assumption1} is satisfied. Thus, if $P$ is subcritical our main results hold true.
 }
\end{example}
An interesting question is to find `minimal' geometric assumptions on $M$ that imply Gaussian estimates of the type \eqref{curvatureestimate}. The upper bound in \eqref{curvatureestimate} is known to be equivalent to a certain Faber-Krahn type inequality (see \cite{greg1,greg}). A well known geometric condition related to  the {\em on-diagonal} lower bound in \eqref{curvatureestimate} is the doubling volume property \eqref{DC}.
In particular, in the next examples we do not assume any a priori curvature assumption on the manifold.
\begin{example}\label{ex3}{\em
Let $(M, g)$ be a complete, connected, noncompact manifold of dimension $N$, and let $ P := -\Delta_{g}$ denote the Laplace-Beltrami operator which satisfy the following properties:

1. For some  $ \ x_{0} \in M$,  there exists $ C > 0$    such that the following doubling volume property holds
\begin{equation}\label{DC}
 \ V(x_{0}, 2r) \leq C V(x_{0}, r)  \qquad  \forall r> 0.
\end{equation}

2. $P$ is subcritical in $M$.

\medskip

3. There exists $C_1 > 0$ such that  the following on-diagonal upper bound estimate holds true
\[
\kP(x_0, x_0, t) \leq \frac{C_1}{V(x_{0}, \sqrt{t})} \qquad   \forall t > 0.
\]
Then by \cite{CA} there exists $c>0$ such that
\[
\kP(x_0, x_0, t) \geq \frac{c}{V(x_{0}, \sqrt{t})} \qquad \forall t>0,
\]
and in particular, there exists $C>0$ such that
\[
\kP(x_0, x_0, t/2) \leq   C \kP(x_0, x_0, t)\qquad \forall t > 0.
\]
 }
\end{example}
\begin{example}\label{ex6}{\em
Let $M$ be a complete, connected, noncompact {\em weighted} Riemannian manifold of dimension $N$. Consider the weighted Laplacian $P$ on $M$, and
denote by  $\kP$ the corresponding heat kernel.
Then the two-sided Gaussian estimates \eqref{curvatureestimate} is equivalent to the validity of the {\em uniform} parabolic Harnack inequality (PHI) (see, \cite{greg1,SC}). We refer to \cite{greg1,SC1,SC} for examples of weighted manifolds satisfying (PHI).
 }
\end{example}
\begin{example}\label{ex4}{\em
 In stochastic processes, the transition density of the random motion naturally leads to the notion of the heat semigroup and hence to the heat kernel. In particular, Dirichlet forms of many families of fractals admit continuous heat kernels that satisfy sub-Gaussian estimates. By a {\em sub-Gaussian kernel} $\tilde{g}$, we mean
 \begin{equation} \label{subgaussian}
\tilde g(x, y, t) := \frac{C}{t^{\frac{\alpha}{\beta}}} \exp{\left( -c\left( \frac{d^{\beta}(x, y)}{t} \right) \right)^{\frac{1}{\beta -1}}},
\end{equation}
where $\alpha > 0,$ $\beta > 1$ are two parameters that come from the geometric properties of the underlying fractal.
The notion of sub-Gaussian estimates was introduced by  M.~T.~Barlow, and E.~A.~Perkins in \cite{BP}. A.~Gregor'yan and A.~Telcs \cite{greg} developed sub-Gaussian estimates for the heat kernel  on metric spaces under suitable assumptions. It follows that complete Riemannian manifolds which admit  two-sided sub-Gaussian estimates for the corresponding heat kernels satisfy our assumption \eqref{assumption1}.
}
\end{example}
We give an example of a manifold with {\em  negative} Ricci curvature, such that our assumption \eqref{assumption1} holds true.
\begin{example}\label{ex5}{\em
Cartan-Hadamard manifolds whose sectional curvatures are bounded above by a \it strictly negative \rm constant, are known to admit a Poincar\'e type (or $L^2$-spectral gap) inequality. Namely, the generalized principal eigenvalue
 \begin{equation*}
 \lambda_{0} = \inf_{{u \in C_{c}^{\infty}(M) \setminus} \{0\}} \frac{\int_{M} |\nabla_{g} u|^2\dv_{g}}{\int_{M} u^{2}\dv_{g}}
 \end{equation*}
is strictly positive.

The  classical example of such a manifold is of course the {\it hyperbolic space}  ${\mathbb H}^N$, where $\lambda_{0} =(N-1)^2/4$.
  Let $ M= \mathbb{H}^{3} $ be the hyperbolic space of dimension 3, then the heat kernel of $ P:= -\Delta_{\mathbb{H}^{3}} - \lambda_{0},$ is given explicitly by
  \[
  \kP(x, y, t) = \left(\frac{1}{4 \pi t} \right)^{-\frac{3}{2}} \frac{d(x, y)}{\sinh d(x, y)}\, \mathrm{e}^{-\frac{d(x, y)^2}{4 t}},
  \]
where $d(x, y)$ denotes the hyperbolic distance between $x$ and $y.$  Hence clearly, $\kP(x, x, \frac{t}{2}) \leq  2^{\frac{3}{2}} \kP(x, x, t)$ holds true for all $t>0$ and $x\in{\mathbb H}^3$ . For higher dimension $N > 3$, the heat kernel of the operator
$P: = -\Delta_{\mathbb{H}^{N}} -  \lambda_{0}$ satisfies
\[
\kP(x, y, t) \!\asymp\!\!  \left(\frac{1}{4 \pi t} \right)^{-\frac{N}{2}}\!\! \mathrm{e}^{- \left[\!  \frac{(N-1) d(x, y)}{2} + \frac{d(x, y)^2}{4t} \right] } \!\!\left(1 + d(x, y) + t  \right)^{\frac{N-3}{2}} \left(1 + d(x, y) \!\right),
\]
 and hence, $\kP(x, x, \frac{t}{2}) \leq  C \kP(x, x, t)$ holds true for all $t>0$ and $x\in{\mathbb H}^N$. Consequently, the results of the present paper hold true for such $P$, and $N\geq 3$. In particular, for any \emph{small perturbation potential} $V$, there exists $\vge_{0} > 0$ such that
 $k_{-\Delta_{\mathbb{H}^{N}} - \vge V}^{\mathbb{H}^{N}} \asymp k_{-\Delta_{\mathbb{H}^{N}}}^{\mathbb{H}^{N}}$ for all $|\vge| < \vge_{0}.$
  }
 \end{example}
 \begin{example}\label{ex7}{\em
Let $P_i$  be a symmetric elliptic operator defined on $M_i$ such that $\lambda_0(P_i,\mathbf{1},M_i)=0,$ where $i=1,2$. Consider the skew product operator $ P:=P_1\times I_1 +I_2\times P_2$ on $M: = M_1\times M_2,$ where $ I_i $  is the identity operator on $ M_i.$ Then
\[
\kP(x, y, t) = k_{P_{1}}^{M}(x_{1}, y_{1}, t) k_{P_{2}}^{M}(x_{2}, y_{2}, t),
\]
where $x=(x_{1}, x_{2}), y=(y_{1}, y_{2})\in M$. If both operators are subcritical and satisfy \eqref{assumption1}, then clearly $P$ is subcritical in $M,$  and its heat kernel satisfies \eqref{assumption1}. Moreover,
if $ P_1 $ is positive-critical in $ M_1,$ and  $ P_2 $ is subcritical in $M_2,$ and its heat kernel $k_{P_{2}}^{M}$ satisfies \eqref{assumption1}, then $P$ is subcritical in $M$,  and by Remark~\ref{rem}, $\kP$ satisfies \eqref{assumption1}. We mention also the case of a twisted tube \cite{GKP} (which is a perturbation of a product space), for which  \eqref{assumption1} is also satisfied.
 }
\end{example}
An anonymous colleague has kindly pointed out to us that our results hold true for the case of universal cover of a compact manifold of negative curvature. Indeed, we have:
\begin{example}\label{ex8}{\em
Let $M$ be the universal cover of a compact manifold of negative curvature. Ledrappier and Lim in \cite{LL} proved recently that the heat kernel of the Laplacian in $M$ satisfies
$$
\lim_{t \rightarrow \infty} t^{\frac{3}{2}} \mathrm{e}^{\lambda_0 t} k_{- \Delta_{g}}^M(x, y, t) = C(x, y),
$$
where $C(x,y)$ is a strictly positive formal eigenfunction of $-\Delta_{g}$ with an eigenvalue $\gl_0$. In particular, the heat kernel of the shifted Laplacian  $P:= -\Delta_g - \lambda_0$ is subcritical in M and satisfies \eqref{assumption1}. Hence, our main results hold true for $P$ on $M$.
}
\end{example}
\subsection{Open problems}\label{subsec_op}
We conclude the paper with some problems that remain open.
\begin{enumerate}
 \item Do theorems~\ref{main} and \ref{spmain} remain true without assuming the doubling condition \eqref{assumption1}? Note that affirmative answers in particular imply that in the class of small perturbations with respect to the heat kernel $\kP$ such that $\gl_0(P,\mathbf{1},M)=0$, the following holds true
     $$S_+(P,V,M)=\{\vge\in \R \mid k_{P_\vge}^M\asymp \kP\}.$$
 \item Prove or disprove Conjecture~\ref{conjequival} in the general  nonsymmetric case.
 \item Study the relationships between the notion of (semi)small perturbations with respect to the Green function and with respect to the heat kernel.
 \item Recall that in the context of (semi)small perturbations with respect to Green functions if $G$ satisfies a certain quasi-metric property,  then the semismallness of a perturbation implies smallness \cite{YP99}. It would be interesting  to find an analogous condition on semismall perturbations with respect to $\kP$ that guarantees smallness.   We remark that, as in the case of small perturbations with respect to Green functions,  we are not aware of any example of a semismall perturbation with respect to a heat kernel which is not a small perturbation.
 \end{enumerate}
Apart from the above open problems related directly to the equivalence of heat kernels, we mention below a far reaching conjecture by M.~Fraas, D.~Krej\v{c}i\v{r}\'{\i}k and Y.~P. regarding the strong ratio limit of the quotients of heat kernels of subcritical and critical operators.  Note that if $P_+$ and $P_0$ are subcritical and critical operators in $M$, respectively, then obviously, $k_{P_+}^M\not\asymp k_{P_0}^M$, and $$\liminf_{t\to\infty}\frac{k_{P_+}^M(x,y,t)}{k_{P_0}^M(x,y,t)}=0.$$
\begin{conjecture}[{\cite[Conjecture~1]{FKP}}]\label{conjMain}
Let $P_+$ and $P_0$ be respectively subcritical and critical
operators in $M$. Then
%
\begin{equation}\label{eqconjMain}
\lim_{t\to\infty}\frac{k_{P_+}^M(x,y,t)}{k_{P_0}^M(x,y,t)}=0
\end{equation}
locally uniformly in $M\times M$.
\end{conjecture}
It follows that for perturbations of the type studied  in the present paper, Conjecture~\ref{conjMain} holds true.
\begin{lem}[{cf. \cite[Theorem~5.4]{FKP}}]\label{thmcond1}
Let $P_0$ be a symmetric critical operator in $M$. Assume that $V=V_+ - V_-$ is a potential such that $V_\pm \geq 0$ and  $P_+:=P_0+V$ is subcritical in $M$.

Assume further that $k_{P_+}^M$ satisfies the $3k$-inequality with respect to $V_-$. Then  there exists a positive constant $C$ such that
\begin{equation}\label{Ass1n}
    k_{P_+}^M(x,y,t)\leq C k_{P_0}^M(x,y,t)\qquad \forall x,y\in M \mbox{ and } t>0.
    \end{equation}
    Moreover, there holds
    \begin{equation}\label{eqconjMain6}
\lim_{t\to\infty}\frac{k_{P_+}^M(x,y,t)}{k_{P_0}^M(x,y,t)}=0,
\end{equation}
locally uniformly in $M \times M.$

In particular, Conjecture~\ref{conjMain} holds true for $P_+:=P_0+V$, where $V$ is any nonzero nonnegative potential.
  \end{lem}
\begin{proof}
By Theorem~\ref{thm_3k_implies_eq} and Lemma~\ref{lem1}, we have $k_{P_+}^M\asymp k_{P_+ + V_-}^M(x,y,t)$.
Note that  $P_+ + V_-=P_0 + V_+$. Therefore, we have
\begin{equation}\label{eq27}
C^{-1} k_{P_+}^M(x,y,t) \leq k_{P_0+V_+}^M(x,y,t)\leq k_{P_0}^M(x,y,t) \qquad \forall x,y\in M \mbox{ and } t>0.
\end{equation}
Using \cite[Theorem~3.1]{FKP}, we conclude that \eqref{eqconjMain6} holds true.
\end{proof}

\medskip

 \begin{center}{\bf Acknowledgments} \end{center}
The authors wish to thank Professor Baptiste Devyver and Professor Alexander Grigor'yan for valuable discussions. They acknowledge the support of the Israel Science Foundation (grants No. 970/15) founded by the Israel Academy of Sciences and Humanities. D.~G. was supported in part at the Technion by a fellowship of the Israel Council for Higher Education.

\end{document}